%% file: congruence.tex
\documentclass[11pt,english]{amsart}

\usepackage{fullpage}
\usepackage[english]{babel}

\usepackage{amscd,amssymb}

\usepackage{hyperref}

\input{macros}

\title{Analytic, Reidemeister and homological torsion for congruence three--manifolds}

\author{Jean Raimbault}
\address{Institut de Math\'ematiques de Toulouse ; UMR5219 \\ Universit\'e de Toulouse ; CNRS \\ UPS IMT, F-31062 Toulouse Cedex 9, France}
\email{Jean.Raimbault@math.univ-toulouse.fr}


\numberwithin{equation}{section}

\begin{document}

\newtheorem{theostar}{Theorem}
\renewcommand*{\thetheostar}{\Alph{theostar}}
\newtheorem{corstar}[theostar]{Corollary}
\newtheorem{propstar}[theostar]{Proposition}
\newtheorem{conjstar}{Conjecture}
\renewcommand*{\theconjstar}{\Alph{conjstar}}
\newtheorem{queststar}[conjstar]{Question}

\newtheorem{theo}{Theorem}[section]
\newtheorem{lemma}[theo]{Lemma}
\newtheorem{prop}[theo]{Proposition}

\begin{abstract}
For a given Bianchi group $\Gamma$ and certain natural coefficent modules $V_\ZZ$ and sequences $\Gamma_n$ of congruence subgroups of $\Gamma$ we give a conjecturally optimal upper bound for the size of the torsion subgroup of $H_1(\Gamma_n;V_\ZZ)$. We also prove limit multiplicity results for the irreducible components of $L_\cusp^2(\Gamma_n\bs\SL_2(\CC))$. 
\end{abstract}

\maketitle

\setcounter{tocdepth}{1}
\tableofcontents
\setcounter{tocdepth}{2}

\section{Introduction}

\input{intro}


\section{Notation and preliminaries} \label{not}

\input{prelim}


\section{Asymptotic geometry of congruence manifolds and approximation of $L^2$-invariants} \label{BS}

\input{BSconv}


\section{Estimates on the logarithmic derivatives of intertwining operators} \label{estent}

\input{entrelacement}


\section{Reidemeister torsion and asymptotic Cheeger-M\"uller equality} \label{CMA}

\input{CMA}


\section{Torsion in (co)homology} \label{endgame}

\input{endgame}


\bibliographystyle{plain}
\bibliography{bib2.bib}

\end{document}

%% file: macros.tex
\newcommand{\eps}{\varepsilon}

\newcommand{\wdt}[1]{\widetilde #1}
\newcommand{\wdh}[1]{\widehat #1}
\newcommand{\ovl}[1]{\overline #1}

\newcommand{\pl}{\partial}
\newcommand{\bs}{\backslash}

\newcommand{\op}{\overset{\circ}{*}}
\newcommand{\vol}{\operatorname{vol}}
\newcommand{\im}{\operatorname{im}}
\newcommand{\tr}{\operatorname{tr}}
\newcommand{\otr}{\operatorname{Tr}}

\newcommand{\Real}{\operatorname{Re}}
\newcommand{\Ima}{\operatorname{Im}}
\newcommand{\sym}{\operatorname{Sym}}
\newcommand{\End}{\operatorname{End}}
\renewcommand{\hom}{\operatorname{Hom}}
\newcommand{\inj}{\operatorname{inj}}
\newcommand{\id}{\operatorname{Id}}

\newcommand{\T}{\mathrm{T}}
\newcommand{\B}{\mathrm{B}}
\newcommand{\N}{\mathrm{N}}

\newcommand{\PSL}{\mathrm{PSL}}

\newcommand{\SL}{\mathrm{SL}}
\newcommand{\SU}{\mathrm{SU}}

\newcommand{\GL}{\mathrm{GL}}

\newcommand{\so}{\mathcal{O}}
\newcommand{\sh}{\mathcal{H}}

\newcommand{\LG}{\mathfrak{g}}

\newcommand{\LP}{\mathfrak{p}}
\newcommand{\LSL}{\mathfrak{sl}}

\newcommand{\frf}{\mathfrak{f}}
\newcommand{\fra}{\mathfrak{A}}
\newcommand{\frb}{\mathfrak{B}}
\newcommand{\frc}{\mathfrak{C}}
\newcommand{\fri}{\mathfrak{I}}

\newcommand{\frj}{\mathfrak{J}}

\newcommand{\abs}{\mathrm{abs}}

\newcommand{\cusp}{\mathrm{cusp}}

\newcommand{\disc}{\mathrm{disc}}
\newcommand{\eis}{\mathrm{Eis}}

\newcommand{\tors}{\mathrm{tors}}
\newcommand{\free}{\mathrm{free}}

\newcommand{\CC}{\mathbb C}
\newcommand{\RR}{\mathbb R}
\newcommand{\ZZ}{\mathbb Z}
\newcommand{\HH}{\mathbb H}
\newcommand{\PP}{\mathbb P}
\newcommand{\QQ}{\mathbb Q}
\newcommand{\FF}{\mathbb F}

\newcommand{\Ade}{\mathbb A}

%% file: intro.tex
\subsection{Torsion in the homology of arithmetic groups and hyperbolic manifolds}

Let $\Gamma$ be a discrete group and $V_\ZZ$ a free, finitely generated $\ZZ$-module with a $\Gamma$-action. The cohomology $H^*(\Gamma; V_\ZZ)$ is an important invariant of $\Gamma$ since it is both accessible to computation (though not necessarily efficiently) and often contains nontrivial information. If $\Gamma$ is the fundamental group of an aspherical manifold $M$ then there is a local system $\mathcal V$ on $M$ such that $H^*(\Gamma; V_\ZZ) = H^*(M; \mathcal V)$. When $M$ is endowed with a Riemannian metric this gives analyticals tools for the study of the characteristic zero cohomology $H^*(\Gamma; V_\CC)$. Maybe the most famous instance of this is when $\Gamma$ is a torsion-free congruence subgroup of $\PSL_2(\ZZ)$ and $V_\ZZ$ is a space of homogeneous polynomials. In this case, by the Eichler--Shimura isomorphism the cohomology can be computed via classical modular forms which correspond to certain harmonic forms on the Riemann surface $\Gamma \bs \HH^2$ (where $\HH^2$ is the hyperbolic plane). More generally, if $\Gamma$ is an arithmetic lattice in a real Lie group then classes in $H^*(\Gamma; V_\CC)$ correspond to ``automorphic forms'' on $G$. This correspondance is interesting in both directions: the analytic side is easier to grasp to prove theoretical results (in particular asymptotic results, as we will see below) but on the other hand the combinatorial side allows for exact computation of the cohomology groups (this has been used for example to experimentally check special cases of Langlands functoriality, as in \cite{GY}).

The torsion part of the cohomology is somewhat less accessible from both the combinatorial and analytic viewpoint. On the other hand it is in certain cases of greater interest than the characteristic zero cohomology. In what follows we will be exclusively interested in arithmetic subgroups of the Lie group $\SL_2(\CC)$. In this case, if $\Gamma$ is torsion-free, it acts freely and properly discontinuously on the hyperbolic space $\HH^3$ and the associated manifold $M = \Gamma \bs \HH^3$ has finite Riemannian volume. It has been observed that very often we have $H^1(M; \CC) = 0$. On the other hand the torsion part tends to be very large. For numerical illustrations of these points see \cite{Sengun}. There is also a form of functoriality for torsion classes which has been explored in \cite{BV}, \cite{Scholze} and \cite{CV}, which makes them of interest in number theory. 

In this paper we will be interested in asymptotic statements about the size of the torsion subgroup of $H^1(\Gamma_n; V_\ZZ)$, when $\Gamma_n$ is a sequence of lattices with covolume tending to infinity. The characteristic zero counterpart of this is the ``limit multiplicity problem'' which was studied by many people and (at least in the case of congruence subgroups) received a definitive solution in \cite{7S} and \cite{FL2}. We will be interested in the following conjecture (we also give a statement for nonarithmetic manifolds since we try to maintain an interest in the topological aspects of the problem). We will use the notion of a ``arithmetic $\Gamma$-module'', that is a lattice $V_\ZZ \subset V$ where $\rho : \SL_2(\CC) \to \GL(V)$ is a representation and $\rho(\Gamma)$ stabilises $V_\ZZ$.\footnote{These exist for nontrivial $\rho$ if and only if $\Gamma$ is arithmetic, on the other hand this includes the case of trivial coefficients which is the most important for topologists.} This conjecture first appeared in print in \cite{BV}, in the arithmetic setting, but in the topological case Thang Le had independently formulated it (and dubbed it ``topological volume conjecture'', in analogy with the Volume Conjecture in quantum topology).

\begin{conjstar} \label{conj}
  If $M = \Gamma \bs \HH^3$ is a closed or cusped hyperbolic 3--manifold then there exists a sequence of subgroups $\Gamma < \Gamma_1 < \cdots < \Gamma_n < \cdots$ with $\bigcap_n \Gamma_n = \{ 1 \}$ and
    \begin{equation} \label{exp_growth}
    \lim_{n \to +\infty} \frac {\log |H^1(\Gamma_n; V_\ZZ)_\tors|}{[\Gamma : \Gamma_n]} = \vol(M) c(V). 
    \end{equation}
    If $\Gamma$ is an arithmetic subgroup of $\SL_2(\CC)$ and $\Gamma_n$ is a sequence of pairwise distinct congruence subgroups of $\Gamma$ then \eqref{exp_growth} holds for $\Gamma_n$. 
\end{conjstar}

The constant $c(V)$ equals $-t^{(2)}(V)$ where $t^{(2)}(V)$ is the $L^2$-torsion associated to $V$, see \cite{BV} or \cite{volI}. For the trivial representation it equals $-1/(6\pi)$, for the adjoint representation $-13/(6\pi)$. It is computed in full generality in \cite{BV}.

The conjecture is completely open for trivial coefficients. There is a certain amount of computational evidence for a positive answer, see the tables in M.H. \c Seng\"un's \cite{Sengun} and the graphs in Section 4 of J. Brock and N. Dunfield's \cite{BrDu}. For some coefficient systems--including the adjoint representation--the limit is proved to hold in \cite{BV} and \cite{7S} for a cocompact lattice $\Gamma$ (see also \cite[Section 6.1]{thesis}). For trivial coefficients the upper bound on the upper limit in \eqref{exp_growth} was established by Thang Le \cite{Le} (the proof is purely topological and hence works for non-necessarily arithmetic lattices). See also \cite{BSV} for some related results in the case of trivial coefficients.

In this paper we will consider the case where $\Gamma$ is a Bianchi group, i.e. there is an imaginary quadratic field $F$ such that $\Gamma=\SL_2(\so_F)$, and deal only with nontrivial coefficients as in \cite{BV}. It is well-known that these groups represent all commensurability classes of arithmetic nonuniform lattices in $\SL_2(\CC)$. We will be concerned in the upper limit in \ref{exp_growth}. We do not manage to deal with all sequences of congruence subgroups of such a $\Gamma$ (see \ref{fail!} below) and we do not adress here the question of dealing with more general sequences of commensurable congruence groups. Also we do not prove that the torsion actually has an exponential growth, which is the most interesting part of the conjecture. This exponential growth---the fact that the limit inferior of the sequence $\log|H_1|/\vol$ is positive---is established for certain sequences by work of the author \cite[Section 6.5]{thesis} and independent work of J. Pfaff \cite{Pfafftors}. However, the method used in the present paper, which is different from those in these two references, gives a clear way to establishing the correct exponential growth rate---it is only because of certain number-theoretical complications that we were not able to get a complet proof. We will explain this in more detail later, for the moment let us state our main theorem. 

\begin{theostar} \label{Main}
Let $\Gamma$ be a Bianchi group, $\Gamma_n$ a cusp-uniform sequence of torsion-free congruence subgroups and $M_n=\Gamma_n\bs\HH^3$. Let $\rho,V$ be a real representation of $\SL_2(\CC)$ which is strongly acyclic and $V_\ZZ$ a lattice in $V$ preserved by $\Gamma$. If $V$ is strongly acyclic then we have 
$$
\lim_{n\to\infty} \left( \frac{\log|H_1(\Gamma_n;V_\ZZ)_\tors|}{\vol M_n} \right) = \lim_{n\to\infty} \left( \frac{\log|H^2(\Gamma_n;V_\ZZ)_\tors|}{\vol M_n} \right) \le -t^{(2)}(V). 
$$
\end{theostar}

Strong acyclicity of representations was introduced in \cite{BV}, it means that the Hodge Laplace operators with coefficients in the local system induced by $\rho$ have a uniform spectral gap for all hyperbolic manifolds; it was shown there to hold for all representations that are not fixed by the Cartan involution of $\SL_2(\CC)$, in particular its nontrivial complex representations. Cusp-uniformity means that the cross-sections of all cusps of all $M_n$ form a relatively compact subset of the moduli space $\PSL_2(\ZZ)\bs\HH^2$ of Euclidean tori. There are obvious sequences of congruence covers which are not cusp-uniform. The proofs of \cite{volI} actually apply not only to cusp-uniform sequences but to all BS-convergent sequences which satisfy a less restrictive condition on the geometry of their cusps, \eqref{square!} below. However even this more relaxed hypothesis fails for some congruence sequences (see \ref{fail!} below) and this raises the question of whether Question \ref{conj} actually has an affirmative answer in these cases. Examples of sequences to which our result does apply include the following congruence subgroups, which are all cusp-uniform (see \ref{cusps} below): 
\begin{equation} \label{ppal}
\begin{split}
\Gamma(\fri) &= \left\{\begin{pmatrix}a&b\\c&d\end{pmatrix}\in \SL_2(\so_F)\: :\: b,c\in\fri,\, a,d\in 1+\fri \right\} \\
\Gamma_1(\fri) &= \left\{\begin{pmatrix}a&b\\c&d\end{pmatrix}\in \SL_2(\so_F)\: : \: c\in\fri,\, a,d\in 1+\fri \right\} \\
\Gamma_0(\fri) &= \left\{\begin{pmatrix}a&b\\c&d\end{pmatrix}\in \SL_2(\so_F),\, c\in\fri \right\}.
\end{split}
\end{equation}
Actually the $\Gamma_0(\fri)$ contain torsion for all $\fri$ but we can apply Theorem \ref{Main} to the sequence $\Gamma_0(\fri)\cap\Gamma'$ where $\Gamma'\subset\Gamma$ is a torsion-free congruence subgroup. As remarked below \ref{orbi} our scheme of proof applies to orbifolds except at one point.


\subsection{A few words about the proof}

Let $\Gamma,\Gamma_n,\rho,V_\ZZ$ be as in the statement of Theorem \ref{Main}. We refer to the introduction of \cite{volI} for more detailed information on how the scheme of proof of \cite{BV} can be adapted to the setting of nonuniform arithmetic lattices. We recall here that the output of Theorems A and B in this paper is that for a Benjamini--Schramm convergent\footnote{See \cite[Definition 1.1]{7S}; we recall that it means that for any $R>0$ the volume of the $R$-thin part $(M_n)_{\le R}$ is an $o(\vol M_n)$.}, cusp-uniform sequence $M_n$ of finite volume hyperbolic three--manifolds, under a certain technical assumption on the continuous part of the spectra of the $M_n$, we have the limit 
\begin{equation} \label{tauabslim}
\lim_{n\to\infty} \frac{\log\tau_\abs(M_n^{Y^n};V)}{\vol M_n} = t^{(2)}(V) 
\end{equation}
where $M_n^{Y^n}$ are obtained from $M_n$ by cutting off the cusps\footnote{One needs to specify how to proceed to choose at which height the cutting is performed, this question is adressed in the quoted paper.} along horospheres at a certain height $Y^n$. The first task we need to complete is to check that the hypotheses of these theorems are satisfied by cusp--uniform sequences of congruence manifolds. Regarding the BS-convergence we prove a quantitative result valid for any sequence of congruence subgroups (Theorem \ref{BSbianchi} below---we recall that $(M_n)_{\le R}$ denotes, as is usual, the $R$-thin part of $M_n$). 

\begin{theostar} \label{BSmain}
Let $\Gamma_n$ be a sequence of congruence of a Bianchi group; then the manifolds $M_n=\Gamma_n\bs\HH^3$ are BS-convergent to $\HH^3$ and we have in fact that there exists a $\delta>0$ such that for all $R>0$
$$
\vol(M_n)_{\le R} \le e^{CR}(\vol M_n)^{1-\delta}
$$
where $C$ depends on $\Gamma$. 
\end{theostar}

We note that this result is much simpler to prove in the cusp-uniform case (see \ref{remBS}). For the estimates on intertwining operators as well our proof goes well for any sequence of congruence subgroups. It is when we study the trace formula along the sequence $M_n$ that our arguments go awry, as some summands of the geometric side seem to diverge as $n\to\infty$. 

Then we want to use \eqref{tauabslim} to study cohomological torsion. The Reidemeister torsion $\tau_\abs(M_n^{Y^n};V)$ is related to the torsion in $H^2(M_n;V_\ZZ)$, in fact it is defined as 
$$
\tau_\abs(M_n^{Y^n};V) = \frac{R^1(M_n^{Y^n})}{R^2(M_n^{Y^n})} \cdot \frac{|H^1(M_n;V_\ZZ)_\tors|}{|H^2(M_n;V_\ZZ)_\tors|}
$$
where $R^p(M_n^{Y^n})$ is the covolume of the lattice $H^p(M_n;V_\ZZ)_\free$ in the space of harmonic forms satisfying absolute conditions on the boundary of $M_n^{Y^n}$. The second thing to be done is to relate these to terms defined on the manifolds $M_n$, using the description of $H^*(M_n;V_\CC)$ by non-cuspidal automorphic forms (namely, harmonic Eisenstein series). We then get a limit 
$$
\lim_{n\to\infty} \frac{\log\tau(M_n;V)}{\vol M_n} = t^{(2)}(V) 
$$
(see \eqref{defRtors} for the precise definition of $\tau$) and it remains to show that the terms $|H^1(M_n;V_\ZZ)_\tors|$ and $R^2(M_n^{Y^n})$ disappear in the limit and that $\liminf R^1(M_n^{Y^n}) \ge 0$. The proofs of these claims use elementary manipulations with the long exact sequence for the Borel-Serre compactification $\ovl M_n$ and its boundary and lemmas on the boundary cohomology. 

This is also where our proof encounters an obstruction to proving the full conjecture, as for the term $R^1(M_n^{Y^n})$ we are not able to get that its limit inferior is positive. For this we would need statements on the integrality of Eisenstein classes which we were not able to establish. We can still isolate a number-theoretical statement which would ensure this as stated in the following proposition. 

\begin{propstar} \label{hypothetic}
  Let $F$ be a quadratic field. For $\chi$ a Hecke character with conductor $\fri_\chi$ we denote by $L(\chi, \cdot)$ the associated L-function and by $L^{\mathrm{alg}}(\chi, \cdot)$ the normalisation which takes algebraic values at half-integers (see \cite[?]{CV}). Assume that there exists $m$ such that
  \[
  \forall \chi, \forall s \in \frac 1 2 \ZZ : |L^{\mathrm{alg}}(\chi, s)|_{\ovl\QQ/\QQ} \le |\fri_\chi|^m.
  \]
  Then we can change the $\limsup$ in Theorem \ref{Main} to a limit. 
\end{propstar}

We will not give a complete proof of this statement here. It is available in the old Arxiv version of this paper (which claims to prove the unconditional statement but in fact proves only this one).


\subsection{Limit multiplicities}

Another problem about sequences of congruence groups is the question of limit multiplicities for unitary representations of $\SL_2(\CC)$. For such a representation $\pi$ on a Hilbert space $\sh_\pi$ and a lattice $\Gamma$ in $\SL_2(\CC)$ one defines its multiplicity $m(\pi,\Gamma)$ to be the largest integer $m$ such that there is a $\SL_2(\CC)$-equivariant embedding of $\sh_\pi^m$ into $L^2(\Gamma\bs\SL_2(\CC))$. The question of limit multiplicities is then to determine the limit of the sequence $m(\pi,\Gamma_n)/\vol(\Gamma_n\bs\HH^3)$ as $\Gamma_n$ ranges over the congruence subgroups of some arithmetic lattice. This question is of particular interest when $\pi$ is a discrete series (when the limit is expected to be positive), and it has been considered in the uniform case by D.L. De George and N.R. Wallach in \cite{dGW}, by G. Savin \cite{Savin} in the nonuniform case. In the case we consider there are no discrete series and thus we expect that the limit multiplicity of any representation will be 0. 

A more precise question to ask is the following: the set $\wdh G$ of irreducible unitary representations (up to isomorphism) of $G=\SL_2(\CC)$ is endowed with a Borel measure $\nu^G$ (the Plancherel measure of Harish-Chandra), and for each lattice $\Gamma\subset G$ the multiplicities in $L^2(\Gamma\bs G)$ define an atomic measure $\nu^\Gamma$. For a congruence sequence $\Gamma_n$ and a Borel set $A\subset \wdh G$, do we have $\nu^{\Gamma_n}(A)\sim\vol(\Gamma_n\bs\HH^3)\nu^G(A)$ as $n$ tends to infinity? In the case where the $\Gamma_n$ are congruence subgroups of a cocompact lattice this is shown to hold in \cite[Section 6]{7S}. The non-compact case is much harder, but T. Finis, E. Lapid et W. M\"uller manage in \cite{FLM} to deal with principal congruence subgroups in all groups $\GL_n/E$, $E$ a number field, and this was generalised to all congruence subgroups in \cite{FL1}, \cite{FL2}. Here we will, much more modestly, deal only with $\SL_2$ over an imaginary quadratic field (note that the first prepublication of this results predates \cite{FL1}). 

\begin{theostar} \label{multlim}
Let $S$ be a regular Borel set in the unitary dual of $G=\SL_2(\CC)$, $\Gamma$ a Bianchi group and $\Gamma_n$ a cusp-uniform sequence of congruence subgroups. Then 
$$
\lim_{n\to\infty} \frac{\sum_{\pi\in S} m(\pi,\Gamma_n)}{\vol M_n} = \nu^G(S)
$$
\end{theostar}

The question of limit multiplicities is related to the growth of Betti numbers in sequences of congruence subgroups via Matsushima's formula and the Hodge-de Rham theorem. The latter has been studied in greater generality (for sequences of finite covers of finite CW-complexes) by W. L\"uck in \cite{Luck2} and M. Farber in \cite{Farber}. Theorem 0.3 of the latter paper together with Theorem 1.12 of \cite{7S} imply that in a sequence of congruence subgroups of an arithmetic lattice the Betti numbers are sublinear in the volume in all degrees except possibly in the middle one where the growth is linear in the volume if the group has discrete series. Another proof of this for cocompact lattices, which actually yields explicit sublinear bounds in certain degrees, is also given in \cite[Section 7]{7S}. For non-compact hyperbolic three--manifolds we dealt with this problem in \cite{volI}; a corollary of \cite[Proposition C]{volI} and of Theorem \ref{BSmain} is then:

\begin{corstar} \label{betti}
Let $\Gamma_n$ be a sequence of torsion-free congruence subgroups of a Bianchi group $\Gamma$, then we have:
\begin{equation*}
\frac{b_1(\Gamma_n)}{\vol M_n} \xrightarrow[n\to\infty]{} 0. 
\end{equation*}
\end{corstar}

Note that one of the proofs given in \cite{volI} is actually a very short and easy argument if one admits \cite[Theorem 1.8]{7S}.


\subsection{Some remarks}

\subsubsection{Sequences that are not cusp-uniform}
\label{fail!}

As noted there the results of \cite{volI} are valid under a slightly less restrictive condition than cusp-uniformity: it suffices that we have 
\begin{equation}
\sum_{j=1}^{h_n} \left(\frac{\alpha_2(\Lambda_{n,j})}{\alpha_1(\Lambda_{n,j})}\right)^2 \le (\vol M_n)^{1-\delta}
\label{square!}
\end{equation}
for some $\delta>0$, where $\Lambda_{n,j}$ are the Euclidean lattices associated to the $h_n$ cusps of $M_n$ (and $\alpha_1,\alpha_2$ are respectively the first and the second minima of the Euclidean norm on a lattice). This is clearly implied by cusp-uniformity in view of Lemma \ref{nbcusps}, and implies the unipotent part of BS-convergence. It is not hard to see that there are examples of congruence sequence which satisfy this condition but are not cusp-uniform. However, there are congruence sequences which do not satisfy \eqref{square!}, for example those associated to the subgroups $K_f^n$ which are the preimage in $K_f = \ovl{\SL_2(\so_F)}$ of $\SL_2(\ZZ/n)$ under the map $K_f\to K_f/K_f(n)\cong\SL_2(\so_F/(n))$: in this case there are $n$ cusps having $\alpha_1\asymp 1$ and $\alpha_2\asymp n$ and the index is about $n^3$. However I have no clue as to whether the limit multiplicities and approximation result should or not be valid for these sequences. 


\subsubsection{Orbifolds}
\label{orbi}
Theorems \ref{multlim} and \ref{BSmain} are valid for sequences of orbifolds as well (see \cite{thesis}). We have not included the necessary additions here in order to keep this paper to a reasonable length and because they are quite straightforward. The approximation for analytic torsion carries to this setting as well, but the Cheeger--M\"uller equality for manifolds with boundary which is one of the main ingredients in the proof of Theorem \ref{CMA} has not, to the best of my knowledge, been proven for orbifolds yet. As for the final steps of the proof of Theorem \ref{Main} they either remain identical (if the cuspidal subgroups are torsion-free) or are simplified by the presence of finite stabilizers for the cusps, which may kill the homology and the continuous spectrum. We will not adress this here, some details are given in \cite{thesis}.


\subsubsection{Trivial coefficients}
For the topologist or the group theorist the trivial local system is the most natural and interesting. The approximation of analytic $L^2$-torsion \cite[Theorem A]{volI}, \cite[Theorem 4.5]{BV} extends to that setting if one assumes that the small eigenvalues on forms have a distribution which is uniformly similar to that of the $L^2$-eigenvalues on $\HH^3$. Let us describe more precisely what this means. Let $M_n$ be a sequence of congruence covers of some arithmetic three--manifold, we know by Theorem \ref{multlim} that for $p=0,1$ the number $m_p([0,\delta];M_n) = (\sum_{\lambda\in[0,\delta]}m_p(\lambda;M_n)$ of eigenvalues of the Laplace operator on $p$-forms on $M_n$ in an interval $[0,\delta]$ behaves asymptotically as $m_p^{(2)}([0,\delta])\vol M_n$ where $m_p^{(2)}$ is the pushforward of the Plancherel measure. We would need to know that we have in fact a uniform decay of $m_p([0,\delta];M_n)/\vol M_n$ as $\delta\to 0$, for example 
\begin{equation}
\frac{m_p([0,\delta];M_n)}{\vol M_n} \le C\delta^c
\label{ptvp}
\end{equation}
for all $\delta>0$ small enough and some absolute $C,c$. We will now describe a (very) idealized situation in which \eqref{ptvp} would hold in a particularly nice form. Let $\alpha_p\in]0,\infty^+]$ be the $p$th Novikov-Shubin invariant of $\HH^3$ (see \cite[Chapters 2 and 5]{Luck}) then there would be an absolute constant $C>0$ such that for any $\delta>0$ small enough and any congruence hyperbolic three--manifold $M_n$ we have 
\begin{equation}
\frac{\sum_{\lambda\in[0,\delta]} m_p(\lambda;M_n)}{\vol M_n} \le C \delta^{\alpha_p}. 
\label{NSF}
\end{equation}
For functions we have $\alpha_0=\infty^+$ (meaning there is a spectral gap on $\HH^3$) and \eqref{ptvp} is known to hold in this case, and in fact in a much more general situation, by L. Clozel's solution of the ``Conjecture $\tau$'' \cite{tau}. For 1-forms $\alpha_1=1$ and one should probably not expect to prove \eqref{NSF} literally (or even for it to hold in this form). We ask the following question, which to the best of our knowledge is wide open (see \cite{Lipnowski_Stern} for some recent advances on the question). 

\begin{queststar}
  Does there exists $\lambda_0>0$ such that for any $\eps>0$ there is a $C_\eps>0$ such that for any congruence hyperbolic three--manifold $M$ and $\delta\le\lambda_0$ we have
  \[
  \frac{\sum_{\lambda\in[0,\delta]} m_1(\lambda;M)}{\vol M} \le C_\eps \delta^{1+\eps}\vol M \quad ?
  \]
(Or less precisely, does this hold for some exponent $c>0$ in place of $1+\eps$ on the right-hand side?)
\end{queststar}

A positive answer to this question is not enough to imply a positive answer to Question \ref{conj} as one still has to analyze the ``regulator'' terms in the Reidemeister torsion: see \cite[9.1]{BV}. In some cases the latter problem is dealt with in work of Bergeron--\c Seng\"un--Venkatesh \cite{BSV}.


\subsubsection{Non-arithmetic manifolds}

If $\Gamma$ is a non-arithmetic lattice in $\SL_2(\CC)$ then it is conjugated into $\SL_2(\so_E[a^{-1}])$ for some number field $E$ and algebraic integer $a$. Thus we can define its congruence covers (whose level will be coprime to $a$), and it can be proved (see \ref{remBS}) that they are BS-convergent to $\HH^3$. Question \ref{conj} still makes sense for trivial coefficients, but is not expected to always have a positive answer in that setting. It is actually expected that for some sequences of non-arithmetic covers with $\inj M_n\to+\infty$ the order of the torsion part of $H_1(M_n;\ZZ)$ does not satisfy \eqref{exp_growth}. We refer to \cite[9.1]{BV} and \cite{BrDu} for more complete discussion around these questions. 


\subsubsection{Sequences of noncommensurable lattices}

Let $M_n$ be a sequence of finite-volume hyperbolic three--manifolds such that $M_n$ BS-converges to $\HH^3$ and their Cheeger constants are uniformly bounded from below. Do we have 
$$
\lim_{n\to+\infty}\frac{\log T(M_n)}{\vol M_n} = \frac 1 {6\pi} \quad ?
$$
Here $T(M_n)$ is the Ray-Singer analytic torsion, regularized as in \cite{Park} if the $M_n$ have cusps. For compact manifolds this is Conjecture 1.4 in \cite{7S}, and for covers \cite[Conjecture 1.13]{BrDu}. Examples of such Benjamini--Schramm convergent sequences are given by noncommensurable arithmetic lattices, for example the sequence of Bianchi groups $\SL_2(\so_F)$ as the discriminant of the field $F$ goes to infinity. See \cite{raimbault, fraczyk}.

In the case of uniform lattices with trace fields having bounded degree it is easy to find natural sequences of coefficient modules of bounded rank satisfying \eqref{exp_growth}. 


\subsection{Outline}

Section \ref{not} introduces the background we use throughout the paper. In Section \ref{BS} we prove Theorem \ref{BSmain}, and in section \ref{estent} we estimate the norm of intertwining operators, thus completing the proof of Theorem \ref{multlim} and of \eqref{tauabslim}. Section \ref{CMA} completes the proof of the asymptotic Cheeger--M\"uller equality between analytic and Reidemeister torsions of the manifolds $M_n$. The final section \ref{endgame} analyses the individual behaviour of the terms in the Reidemeister torsion, finishing the proof of Theorem \ref{Main}.


\subsection{Acknowledgments}

This paper originates from my Ph.D. thesis \cite{thesis}, which was done under the supervision of Nicolas Bergeron, and it is a pleasure to thank him again for his guidance. During the completion of this work I was the recipient of a doctoral grant from the Universit\'e Pierre et Marie Curie. This version was mostly worked out during a stay at Stanford University and reworked while I was a postdoc at the Max-Planck Institut in Bonn. I am grateful to these institutions for their hospitality and to Akshay Venkatesh for inviting me and for various discussions around the topic of this paper.

%% file: prelim.tex
\subsection{Bianchi groups and congruence manifolds}

For this section we fix an imaginary quadratic field $F$ and let 
$\Gamma=\SL_2(\so_F)$ be the associated Bianchi group. We will denote by $\Ade_f$ the ring of finite ad\`eles of $F$. At infinity we fix the maximal compact subgroup $K_\infty$ of $\SL_2(\CC)$ to be $\SU(2)$; if $K_f'$ is a compact-open subgroup of $\SL_2(\Ade_f)$ we will adopt the convention of denoting by $K'$ the compact-open subgroup $K_\infty K_f'$ of $K_\infty\SL_2(\Ade_f)$.

\subsubsection{Congruence subgroups}

For any finite place $v$ of $F$ let $K_v$ be the closure of $\Gamma$ in $\SL_2(F_v)$ ; then $K_f = \prod_v K_v$ is the closure of $\Gamma$ in $\SL_2(\Ade_f)$. A congruence subgroup of $\Gamma$ is defined to be the intersection $\Gamma\cap K_f'$ where $K_f'$ is a compact-open subgroup of $K_f$. Let $\Gamma(\fri)$ be defined by \eqref{ppal} and $K_f(\fri)$ its closure in $K_f$; then $\Gamma(\fri)=\Gamma\cap K_f(\fri)$ so that $\Gamma(\fri)$ is indeed a congruence subgroup ; likewise, $\Gamma_0$ and $\Gamma_1(\fri)$ are ``congruence-closed'', i.e. they are equal to the intersection of their closure with $\SL_2(F)$. 

For a compact-open subgroup $K_f'\subset K_f$ we will denote by $\Gamma_{K'}=\SL_2(F)\cap K_f'$ the associated congruence lattice ; we define its level to be the largest $\fri$ such that $K_f(\fri)\subset K_f'$\footnote{This is well--defined as $K_f(\fri)K_f(\fri')=K_f(\frj)$ where $\frj=\gcd(\fri,\fri')$.}. Then the following fact is well-known (see \cite[Lemme 5.8]{thesis}).

\begin{lemma}
For any compact-open $K_f'\subset K_f$ we have 
$$
[K_f:K_f'] \ge \frac 1 3 |\fri|^{\frac 1 3}
$$
where $\fri$ is the level of $K_f'$. 
\label{index-level}
\end{lemma}


\subsubsection{Congruence manifolds}

For a compact-open $K_f'$ we denote by $M_{K'}$ the orbifold $\Gamma_{K'}\bs\HH^3$. The strong approximation theorem for $\SL_2$ (which in this case is a rather direct consequence of the Chinese remainder theorem) states that the subgroup $\SL_2(F)$ is dense in $\SL_2(\Ade_f)$, and it follows that we have an homeomorphism
\begin{equation}
M_{K'} \cong \SL_2(F)\bs\SL_2(\Ade)/K'.
\label{stra}
\end{equation}


\subsubsection{Unipotent subgroups} 
\label{unipotent}
Let $\N$ be a unipotent subgroup of $\SL_2$ defined over $F$, $\B$ its normalizer in $\SL_2$ and $\alpha$ the morphism from $\B/\N$ to the multiplicative group\footnote{Which is isomorphic to the automorphism group of $\N$ since the latter is itself isomorphic to the additive group.} given by the conjugacy action on $\N$. For any place $v$ of $F$ we have the Iwasawa decomposition 
$$
\SL_2(F_v) = \B(F_v) K_v , 
$$
and this yields also that $\SL_2(\Ade) = \B(\Ade) K$. We define a height function on $\SL_2(\Ade)$ by 
\begin{equation}
y(g) = \max \{|\alpha(b)| ,\, b\in\B(\Ade) \: :\: \exists \gamma\in\SL_2(F),k\in K, \, \gamma g = bk \}, 
\label{height}
\end{equation}
this does not depend on the $F$-rational unipotent subgroup $\N$. We fix the unipotent subgroup $_0\N$ to be the stabilizer of the point $(0,1)$ in affine 2-space, and we identify it with the additive group using the isomorphism $\psi$ sending 1 to the matrix $\begin{pmatrix} 1&1\\0&1\end{pmatrix}$. Then $\N$ is conjugated to $_0\N$ by some $g\in\SL_2(F)$ and we identify $\N$ with the additive group using the isomorphism given by the composition of conjugation by $g$ with $\psi\circ\alpha(b)^{-1}$ where $g=bk$ (and we view $\alpha(b)$ as an automorphism of the additive group).


\subsubsection{Cusps}
\label{cusps}

The cusps of the manifold $M_{K'}$ are isometric to the quotient $\B(F)\bs\SL_2(\Ade)/K'$; in particular the number $h$ of cusps is equal to the cardinality of the finite set 
$$
\mathcal{C}(K') = \B(F)\B(F_\infty)\bs\SL_2(\Ade)/K'.
$$ 
We can describe accurately the cross-section of each cusp.Let $\N$ be the unipotent subgroup which is the commutator of the stabilizer in $\SL_2$ of the point $(a:b)\in\PP^1(F)$. We may suppose that $a,b$ have no common divisor in $\so_F$ except for units, then the ideal $(a,b)$ is equal to some ideal $\frc$ without principal factors and we write $(a)=\frc\fra$, $(b)=\frc\frb$, so that $(\fra,\frb)=1$. Then we have \cite[Proposition 5.1]{thesis}:
$$
K_f \cap \N(F) = 1 + \left\{ \begin{pmatrix} \frac ab c & -\frac{a^2}{b^2} c \\ c & -\frac ab c\end{pmatrix},\: c\in\frb^2\right\} = 1 + \left\{ \begin{pmatrix} -\frac ba c &  c \\ -\frac{b^2}{a^2}c & \frac ba c\end{pmatrix},\: c\in\fra^2\right\}.
$$
In particular, since the ideals of $\so_F$ are a uniform family of lattices it follows easily that the families $\Gamma(\fri),\Gamma_0(\fri)$ or $\Gamma_1(\fri)$ are cusp-uniform (see \cite[Lemme 5.7]{thesis}).


\subsection{Analysis on $\SL_2(\Ade)$ and Eisenstein series} \label{spec_dec}

We fix a Borel subgroup $\B\subset\SL_2$ defined over $F$ and a maximal split torus $\T$ in $\B$, and let $\N$ be the unipotent radical of $\B$. 

\subsubsection{Haar measures}
\label{haar}

We fix the additive and multplictive measures on each $F_v$ and $F_v^\times$ as usual, and choose the Haar measure of total mass one on $K_v$. We take the Haar measure on $\B(F_v) = \N(F_v)\T(F_v) \cong F_v\rtimes F_v^\times$ to be given by $d(n_va_v)=\frac{dx_v}{|.|_v}\otimes d^\times x_v$. On a proper quotient we always take the pushforward measure, in particular the measure on $\SL_2(\Ade)$ is the pushforward of $dbdk$ in the Iwasawa decomposition $\SL_2(\Ade)=\B(\Ade)K$. 

\subsubsection{Spaces of functions} 

Let $A^1$ be the subgroup of $\T(\Ade)$ such that every character from $\T(\Ade)$ to $\RR_+^\times$ factors through $A^1$ ; then $\T(F)\cong F^\times$ is contained in $A^1$ and we will denote by $\sh$ the Hilbert space $L^2(F^\times\bs(A^1K))$. We have a natural isomorphism
$$ 
\sh \cong \CC[C(F)]\otimes L^2(K)  
$$
where $C(F)$ is the class-group of $F$. Let $\chi$ be a Hecke character, then it induces characters of $A^1$ and $\B(\Ade)$ that we will continue to denote by $\chi$. For $s\in\CC$ and $\phi\in\sh\cap C^\infty(A^1K)$ there is a unique extension $\phi_s$ of $\phi$ to $\SL_2(\Ade)$ which satisfies:
\begin{equation}
\forall n\in N(\Ade), a_\infty\in A_\infty, a\in A^1,k\in K \: \phi_s(na_\infty ak)=|\alpha(a_\infty)|_\infty^s\phi(ak).
\label{prolongement}
\end{equation}
We will denote by $\sh_s$ the space of such extensions, and $\sh_s(\chi)$ its subspace of functions having $A^1$-type $\chi$ on the left. The space $\sh_s$ is acted upon by $\SL_2(\Ade)$ by right translation; when $s\not= 0,1$ the decomposition of $\sh_s$ into $\SL_2(\Ade)$-irreducible factors is given by the $\sh_s(\chi)$. If $\tau$ is a finite-dimensional complex continuous representation of $K$ we define $\sh_s(\chi,\tau)$ to be the subspace of $\sh_s(\chi)$ containing the functions which have $K$-type $\tau$ on the right (in other words, the projection to $\sh_s(\chi)$ of the subspace of $K$-invariant vectors in $\sh_s(\chi)\otimes\tau$). 


\subsubsection{Eisenstein series}

\label{defeis}

For a function $f\in C^\infty(\B(F)\N(\Ade)\bs\SL_2(\Ade))$ put
\begin{equation}
E(f)(g)=\sum_{\gamma\in \SL_2(F)/\B(F)}f(\gamma^{-1}g).
\end{equation}
The height function $y$ defined by \eqref{height} is left $\B(F)\N(\Ade)$-invariant, and it is well-known (cf. \cite[Lemme 5.23]{thesis}) that the series 
$$
E(y^s)(g)\sum_{\gamma\in \SL_2(F)/\B(F)} y(\gamma^{-1}g)^s
$$
converges absolutely for all $g\in\SL_2(\Ade)$ and $\Real(s)>2$, uniformly on compact sets. For $\phi\in\sh$ the function $\phi_s$ defined through \eqref{prolongement} we denote $E(\phi_s)=E(\phi,s)$ which is convergent for $\Real(s)>2$ according to the above. We have the following fundamental result, due to A. Selberg for $\SL_2/\QQ$ and to R. Langlands in all generality (see \cite[5.4.1]{thesis} for a simpler proof in this case, based on ideas of R. Godement \cite{Godement}).

\begin{prop}
The function $s\mapsto E(\phi,s)$ has a meromorphic continuation to $\CC$, which is homolorphic everywhere if $\chi\not=1$. If $\chi=1$ there is only one pole of order one at $s=1$.
\label{meromorphe}
\end{prop}

The main point of the theory of Eisenstein series is that they give the orthogonal complement to the discrete part of the regular representation on $L^2(\SL_2(F)\bs\SL_2(\Ade)$. The space $L_\cusp^2(\SL_2(F)\bs\SL_2(\Ade))$ of cusp forms is usually defined to be the closed subspace of all functions on $\SL_2(F)\bs\SL_2(\Ade)$ whose constant term (defined by \eqref{ct} below) vanishes. 

\begin{prop}
The map 
$$
\int_{-\infty}^{+\infty} \mathcal{H}_{\frac 1 2+iu}(\chi) \frac{du}{2\pi}\ni\psi \mapsto E(\psi)\in L^2(\SL_2(F)\bs\SL_2(\Ade))
$$
is an isometry onto the orthogonal of the space $L_\cusp^2(\SL_2(F)\bs\SL_2(\Ade))\oplus\CC$. Moreover, for any $\phi\in C_c^\infty(\SL_2(\Ade)$ the associated operator on $L_\cusp^2(\SL_2(F)\bs\SL_2(\Ade))$ is trace-class; in particular $L_\cusp^2(\SL_2(F)\bs\SL_2(\Ade))$ decomposes as a Hilbert sum of irreducible, $\SL_2(\Ade)$-invariant closed subspaces. 
\end{prop}


\subsubsection{Intertwining operators}

Let $f$ be a continuous function on $SL_2(F)\backslash \SL_2(\Ade)$. We define its constant term to be 
\begin{equation}
f_P(g)=\int_{\N(F)\bs\N(\Ade)}f(ng)dn.
\label{ct}
\end{equation}
Let $\phi\in\sh$, $\phi_s$ be defined by \eqref{prolongement} and $f=E(\phi_s)$. We use the Bruhat decomposition $\SL_2(F)/\B(F)=\{\B(F)\}\cup \{\gamma w\B(F), \gamma\in N(F)\}$ and when $\Real(s)>3/2$ we get:
\begin{eqnarray*}
f_P(g)& = &\int_{\N(F)\bs \N(\Ade)}\sum_{\gamma\in \SL_2(F)/\B(F)}\phi_s(\gamma^{-1}ng)dn\\
      & = &\int_{\N(F)\bs \N(\Ade)}\phi_s(ng)dn+\sum_{\gamma\in \N(F)}\int_{\N(F)\bs \N(\Ade)}\phi_s(w\gamma ng)dn\\
      & = &\phi_s(g)+\int_{\N(\Ade)}\phi_s(wng)dn.
\end{eqnarray*}
and we define the intertwining operator $\Psi(s)$ on $\sh$ by
\begin{equation}
(\Psi(s)\phi)(ak)=\int_{N(\Ade)}\phi_s(wnak)dn \, . 
\label{intertwine}
\end{equation}
We obtain (using the notation of \eqref{prolongement}):
\begin{equation}
E(\phi_s)_P=\phi_s + (\Psi(s)\phi)_{1-s}. 
\label{tceis}
\end{equation}
One can check that $\Psi(s)$ induces a $\SL_2(\Ade)$-equivariant endomorphism  on $\sh_s$, which sends the irreducible subspace $\sh_s(\chi)$ to $\sh_s(\chi^{-1})$. For $\Real(s)=\frac 1 2$ the map $\Psi(s)$ is an isometry for the inner product of $\sh$.


\subsubsection{Maass-Selberg}
Finally we record the Maass-Selberg expansions \cite[5.4.4]{thesis}; for $s\in\CC-\RR\cup(\frac 1 2+i\RR)$:
\begin{equation}
\begin{split}
&\langle T^YE(s,\phi),T^YE(s',\psi)\rangle_{L^2(G(F)\bs G(\Ade))} =\\ 
    &\qquad \frac 1{2(s+\ovl{s'}-1)}(Y^{2(s+\overline{s'}-1)}\langle\phi,\psi\rangle_{\mathcal{H}} - Y^{-2(s+\ovl{s'}-1)}\langle\Psi(s')^*\Psi(s)\phi,\psi\rangle_{\mathcal{H}}) \\
      &\qquad + \frac 1{2(s-\ovl{s'})}(Y^{2(s-\overline{s'})}\langle\phi,\Psi(s')\psi\rangle_{\mathcal{H}} - Y^{2(-s+\ovl{s'})}\langle\Psi(s)\phi,\psi\rangle_{\mathcal{H}}).
\end{split}
\label{preMaass-Selberg}
\end{equation}
When $s\in\RR$ this degenerates to 
\begin{equation}
\begin{split}
\langle T^Y E(\phi,s),T^Y E(\psi,s)\rangle_{L^2(G(F)\bs G(\Ade))}^2 =& \frac{Y^{4s-2}}{4s-2}\langle\phi,\psi \rangle_{\mathcal{H}} - \frac{Y^{-4s+2}}{4s-2}\langle \Psi(s)\phi,\Psi(s)\psi \rangle_{\mathcal{H}} \\
       & + \log Y \langle \Psi(s)\phi,\psi\rangle_{\mathcal{H}} + \left\langle\frac{d\Psi(s+iu)}{du}|_{u=0}\phi,\psi\right\rangle_{\mathcal{H}}. 
\end{split}
\label{Maass-Selberg-real}
\end{equation}


\subsection{Regularized traces}

\subsubsection{Differential forms and $L^2(\SL_2(F)\backslash \SL_2(\Ade))$}
\label{passage}

Let $K'$ be a compact-open subgroup of $K$, and $\rho$ a representation of $\SL_2(\CC)$ on a finite-dimensional real vector space $V$. We can associate to $\rho$ a local system on $M=M_{K'}$ and we denote by $\Omega^p,L^2\Omega^p(M;V)$ the spaces of smooth and square-integrable $p$-forms on $M$ with coefficients in $V$ (see \cite [2.3]{volI}). It is well-known that there is an identification
\begin{equation}
\left(L^2(\SL_2(F)\bs \SL_2(\Ade))\otimes \wedge^p\LP\otimes V\right)^{K'} \to L^2\Omega^p(M;V_\CC). 
\label{id1}
\end{equation}
Let $\tau$ be the representation of $K$ who has for it's finite part $\CC[K/K']$ and whose infinite part is equal to the representation of $K_\infty$ on $V_\CC\otimes\wedge^p\LP^*$. We define the map $E(s,.)$ to be such that the following diagram commutes:
\begin{equation}
\begin{CD}
\left(\sh_s \otimes V_\CC\otimes\wedge^p\LP^*\right)^{K'}   @>E>>       \left(L^2(\SL_2(F)\bs \SL_2(\Ade))\otimes V_\CC\otimes\wedge^p\LP^* \right)^{K'} \\       
   @|                                                    @|               \\
\CC^h\otimes V_\CC\otimes\wedge^p\LP^*   @>E(s,.)>>  L^2\Omega^p(\Gamma_{K'}\bs\HH^3 ; V_\CC)
\end{CD}
\label{eis}
\end{equation}
where $h$ is the number of cusps of $M$ and we identify $\B(F)A_\infty \N(\Ade)\bs \SL_2(\Ade)/K_f'$ with $\CC^h$. For $p=0,1$ we retrieve the maps defined in \cite[3.1.3 and 3.1.4]{volI} which associate to a section or 1-form on $\pl\ovl M$ an element of $L^2\Omega^p(\Gamma_H\bs\HH^3 ; V_\CC)$.


\subsubsection{Spectral trace}

Let $M$ be a congruence hyperbolic three--manifold and $\Delta^p[M]$ the Hodge Laplacian on $p$-forms on $M$ with coefficients in $V_\CC$. We recall the ``spectral'' definition of the regularized trace given in \cite[3.2.4]{volI}. Let $p=0$ for now, and let $\phi$ be a function on $\RR$ such that the associated automorphic kernel $K_\phi^0$ on $\HH^3\times\HH^3$ has compact support (see \cite[Section 3.2]{volI}). The regularized trace of $\phi(\Delta^0[M])$ is given by 
\begin{equation}
\begin{split}
\otr_R \phi(\Delta^0[M]) &= \sum_{j\ge 0} m(\lambda_j;M)\phi(\lambda_j)   + \frac 1 4 \sum_{l=-2q}^{2q}d_l \phi\left(-l^2 + 4 + \lambda_V\right) \tr\Psi_l(0) \\
                      &\quad - \frac 1{2\pi} \int_{-\infty}^{+\infty} \sum_{l=-2q}^{2q} d_l \phi\left(-u^2 + 4 - l^2 + \lambda_V \right)\tr\left(\Psi_l(iu)^{-1}\frac{d\Psi_l(iu)}{du}\right) du.
\end{split}
\label{spec}
\end{equation}
where:
\begin{itemize}
\item The $\lambda_j,j\ge 0$ are the eigenvalues of $\Delta^0[M]$ in $L^2(M;V_\CC)$;
\item For $\lambda\in[0,+\infty[$, $m(\lambda;M) = \dim\ker(\Delta^0[M]-\lambda\id)$ ;
\item For $u\in\RR$, $\Psi_l(iu)$ is the operator on $\CC^h\otimes W_l$ corresponding to $\Psi(\frac 1 2(1+iu))$ under the identifications on the right side of \eqref{eis};
\item $W_l$ is a certain subspace of $V_\CC$, and we have $V_\CC=\bigoplus_l W_l$. 
\end{itemize}
For more details see \cite[Section 3.1]{volI}. We will skip the definition for 1-forms since it is basically the same (see loc. cit.). 


\subsection{Homology and cohomology}

Here we consider a CW-complex $X$, $\Lambda=\pi_1(X)$ and $L$ a free $\ZZ$-module of finite rank with a $\Lambda$-action. There are then defined chain an cochain complexes $C_*(X;L),d_*$ and $C^*(X;L),d^*$. If $X$ is aspherical then $H_*(X;L)\cong H_*(\Lambda;L)$ and $H^*(X;L)\cong H^*(\Lambda;L)$. 

\subsubsection{Kronecker pairing}

\label{kronecker}

Let $L^*$ be the dual $\hom(L,\ZZ)$, $Z_p(X;L)=\ker(d_p)$ et $Z^p(X;L^*)=\ker(d^p)$. There is a natural bilinear form on $Z_p(X;L)\times Z^p(X;L^*)$ which induces a nondegenerate bilinear form 
$$
(.,.)_X : H_1(X;L)_\free\times H^1(X;L^*)_\free\to\ZZ. 
$$
If $Y$ is a sub-CW-complex of $X$ and $i$ its inclusion in $X$ we have the following property
\begin{equation}
\forall \eta\in H_p(Y;L^*),\omega\in H^p(X;L): \quad  (i_*\eta,\omega)_X = (\eta,i^*\omega)_Y.
\label{naturel}
\end{equation}
In case there is a perfect duality between two $\Lambda$-modules $L,L'$ which comes from a $\Lambda$-invariant bilinear form we get a Kronecker pairing on $H_p(X;L)_\free\times H^p(X;L')_\free$ satisfying \eqref{naturel}; for $\Gamma$ a lattice in $\SL_2(\CC)$ there exists such a self-duality for the $\Gamma$-modules $V=V_{n_1,n_2}$ (given by the form induced on $V$ by the determinant pairing on $V$). This form bilinear form, which we will denote by $\langle\cdot,\cdot\rangle_V$, is actually defined over $\ZZ[m^{-1}]$ for $m= n_1! n_2!$, and if $V_\ZZ\subset V_\QQ$ is a lattice we put
$$
V_\ZZ' = \{v'\in V_\QQ\: : \:\forall v\in V_\ZZ,\, \langle v',v\rangle_V\in\ZZ \}
$$
which is another lattice in $V_\QQ$.


\subsubsection{Poincar\'e duality}

We suppose now that $X$ is a $n$-dimensional compact manifold with boundary $\pl X$. Poincar\'e duality is an isomorphism of graded $\ZZ$-modules  $H_*(X;V)\xrightarrow[]{\sim} H^{n-*}(X,\partial X ; V)$ or $H_*(X,\partial X;V)\xrightarrow[]{\sim} H^{n-*}(X; V)$. It is compatible with the long exact sequences of the pair $X,\pl X$ in the following sense \cite[Theorem V.9.3]{Bredon}.

\begin{prop}
The diagram:
\[\begin{CD}
\ldots @>>> H_p(\partial\overline M; V) @>>> H_p(M; V) @>>> H_p(\overline M,\partial\overline M; V) @>>> \ldots \\
  &      &          @VVV                     @VVV             @VVV\\
\ldots @>>> H^{n-1-p}(\partial\overline M; V^*) @>>> H^{n-p}(\overline M,\partial\overline M; V^*) @>>> H^{n-p}(M; V^*) @<<< \ldots
\end{CD}\]
is commutative, where horizontal lines are the long exact sequences in homology and cohomology of $X,\pl X$ and vertical arrows are Poincar\'e duality morphisms. 
\label{pd}
\end{prop}

%% file: BSconv.tex
In this section we shall, assuming the results of the next section, prove Theorem \ref{multlim} from the introduction and the approximation result for analytic torsion (Theorem \ref{analytic} below).

\subsection{Benjamini-Schramm convergence}

The following result generalizes \cite[Theorem 1.12]{7S} to the case of noncompact congruence subgroups of $\SL_2(\CC)$. 

\begin{theo}
There are $\delta,c>0$ such that for any Bianchi group $\Gamma=\Gamma(\so_F)$ and sequence $\Gamma_n$ of torsion-free congruence subgroups in $\Gamma$, for all $R>0$ we have
$$
\vol\{x\in M_n \: :\: \inj_x(M_n)\le R \} \le e^{cR}[\Gamma:\Gamma_n]^{1-\delta}
$$
In particular, the sequence of hyperbolic manifolds $M_n=\Gamma_n\bs\HH^3$ is BS-convergent to $\HH^3$.
\label{BSbianchi}
\end{theo}

We record the following much weaker consequence (see Lemma \ref{BSconv} below) of this as a separate fact; note that this is actually the only part of Theorem \ref{BSbianchi} that we make full use of here, and a direct proof is much easier than that of the latter. 

\begin{lemma}
Notation as above we have 
$$
h_n \ll [\Gamma:\Gamma_n]^{1-\delta}
$$
where $h_n$ is the number of cusps of $M_n$. 
\label{nbcusps}
\end{lemma}

Recall that $K_f$ is the closure in  $\SL_2(\Ade_f)$ of $\Gamma=\SL_2(\so_F)$ 
and let $K_f'$ be a closed finite-index subgroup with level $\fri$. We will 
show that for the subgroup $\Gamma'=\Gamma_{K'}$ of $\Gamma$ and 
$M=\Gamma'\bs\HH^3$ we have 
$\vol\{x\in M \: :\: \inj_x(M)\le R \} \le C[K:K']^{1-\delta}$. 


\subsubsection{Remarks} \label{remBS}
1. The mere BS-convergence (without the precise estimates) follows from \cite[Theorem 1.11]{7S}: one can see that it implies that any invariant random subgroup which is a limit of a sequence of congruence covers has to be supported on unipotent subgroups, which is impossible if the limit is nontrivial (for example because of ``Borel's density theorem'' \cite[Theorem 2.9]{7S}). 

2. As a corollary we get that there is an $\eps>0$ such that
$$
\vol M_{\le\eps\log\vol M} \le (\vol M)^{1-\delta} 
$$
for all (manifold) congruence covers $M$ of a given Bianchi orbifold.


\subsubsection{Benjamini--Schramm convergence of manifolds with cusps}
\label{BSredux}

We recall some notation: for a hyperbolic manifold $M$ we let $N_R(M)$ be the number of closed geodesics of length less than $R$ on $M$. If $\Lambda$ is a lattice in $\CC$ we define 
$$
\alpha_1(\Lambda) = \min\{|v| \: :\: v\in\Lambda, v\not=0 \}
$$
and for any $v_1\in\Lambda$ such that $|v_1|=\alpha_1(\Lambda)$
$$
\alpha_2(\Lambda) = \min\{|v| \: :\: v\in\Lambda, v\not\in \ZZ v_1 \}.
$$
Then the ratio $\alpha_2/\alpha_1$ only depends on the conformal class of $\Lambda$, in particular if $\Gamma\not\ni -1$ is a lattice in $\SL_2(\CC)$ and $N$ a unipotent subgroup such that $\Gamma\cap N$ is nontrivial (we will say that $N$ is $\Gamma$-rational) then $\alpha_2/\alpha_1(\Gamma\cap N)$ is well-defined and depends only on the $\Gamma$-conjugacy class of $N$. We can then estimate the volume of the thin part as follows (in particular, to prove that a sequence of finite covers of a fixed orbifold is BS-convergent we need only give $o(\vol M_n)$-bounds for the right-hand side).

\begin{lemma}
Let $M=\Gamma\bs\HH^3$ be a finite--volume hyperbolic three--manifold and let $N_1,\ldots,N_{h_n}$ be representatives for the $\Gamma$-conjugacy classes of unipotent subgroups. Put $\Lambda_{n,j}=\Gamma_n\cap N_j$, then there are constants $C$ (depending on $\Gamma$)  and $c>0$ such that 
$$
\vol (M_n)_{\le R} \le e^{cR} \left( RN_R(M_n) + \sum_{j=1}^{h_n} \frac {\alpha_2(\Lambda_{n,j})} {\alpha_1(\Lambda_{n,j})} \right).
$$
\label{BSconv}
\end{lemma}

\begin{proof}
This follows from the two following facts:
\begin{itemize}
\item[(i)] If $g\in\SL_2(\CC)$ is loxodromic the $R$-thin part of $\langle g\rangle\bs\HH^3$ has volume $\le C\ell e^{cR}$, where $C$ depends only on the minimal translation length $\ell$ of $g$;
\item[(ii)] If $\Lambda$ is a lattice in a unipotent subgroup $N$ of $\SL_2(\CC)$ then the $R$-thin part of $\Lambda\bs\HH^3$ is of volume $\le e^{cR}\alpha_1(\Lambda)/\alpha_2(\Lambda)$. 
\end{itemize}
The point (i) follows immediately from the fact that if $L$ is the axis of $g$, then $d(x,gx)\ll R$ with a constant independent of $g$, and the gact that the volume of a $R$-neighbourhood of the closed geodesic in $\langle g\rangle\bs\HH^3$ is of volume $\le \ell e^{cR}$ (where $e^{cR}$ is an upper bound for the volume of a radius $R$ ball in $\HH^3$). 

For point (ii) we observe that we can parametrize $\Lambda\bs\HH^3$ as $T\times[0,+\infty[$, where $T$ is the Euclidean torus $\Lambda\bs\CC$, which we suppose normalized so that $\alpha_1(\Lambda)=1$ (we conformally identify $N$ with $\CC$) and the product metric $(dx^2+dy^2)/y^2$. Then the $R$-thin part is contained in $T\times[e^{-cR},+\infty[$ (where $c$ is such that if $x,y\in\HH^3$ belong to an horosphere $H$ with Euclidean distance $d_H$, we have $d_{\HH^3}(x,y)\ge \log(1+d_H(x,y))$). The volume of the latter is easily seen to be $\le \vol(T)e^{cR} \ll e^{cR}\alpha_2(\Lambda)$. 
\end{proof}


\subsubsection{Proof of Theorem \ref{BSbianchi}, the loxodromic part}

Here we recall how the bound on $N_R$ follows from the results in \cite[Section 5]{7S}. Let $c$ be a closed geodesic in the orbifold $\Gamma\bs\HH^3$ and $\gamma\in\Gamma$ any element of the associated loxodromic conjugacy class in $\Gamma$. For any $g\in K_f$ we have that $\gamma$ fixes the coset $gK_f'$ if and only if $g\gamma g^{-1}$ belongs to $K_f'$, so that the number of lifts of $c$ in $M_{K'}$ is equal to the number of fixed points of $\gamma$ in $K_f/K_f'$. By Theorem 1.11 in \cite{7S} there are constants $\delta$ (depending on $F$) and $C$ (depending on $c$) such that the latter is less than $C|K_f/K_f'|^{1-\delta}$. This shows that for a given $R$ there is a $C_R$ such that for all $K'$ we have 
$$
N_R(M_{K'}) \le C_R|K_f/K_f'|^{1-\delta} = \left(C_R\vol(\Gamma\bs\HH^3)\right) (\vol M_{K'})^{1-\delta}. 
$$


\subsubsection{Proof of Theorem \ref{BSbianchi}, the unipotent part}
Now we have to bound the second term in Lemma \ref{BSconv}: we want to show that 
\begin{equation}
\sum_{j=1}^{h_{K'}} \frac {\alpha_2(\Lambda_{K',j})} {\alpha_1(\Lambda_{K',j})} \ll |K_f/K_f'|^{1-\delta} 
\label{Scusp}
\end{equation}
where $\Lambda_{K',j}= {}_j\N(F)\cap K_f'$ where the $_j\N$ are representatives of the $\Gamma_{K'}$-conjugacy classes of unipotent subgroups in $\SL_2(F)$. We fix a unipotent subgroup $\N$ in $\SL_2(F)$, and let $N=\N(\Ade_f)\cap K_f$; clearly it suffices to prove that \eqref{Scusp} holds if we sum only over the unipotent groups contained in the $\Gamma$-conjugacy class of $\N(F)$. These $\Gamma_{K'}$-conjugacy classes are in natural bijection with the set of double cosets $N\bs K_f/K_f'$

For $p\in\ZZ$ a rational prime we will denote $F_p=\QQ_p\otimes_\QQ F$ and $K_p$ the closure of $\Gamma$ in $\SL_2(F_p)$. The latter is isomorphic to 
\begin{itemize}
\item $\SL_2(\so_v)$ in case $p$ is inert or ramified in $F$ and $v$ is the 
corresponding place of $F$;
\item $\SL_2(\ZZ_p)\times\SL_2(\ZZ_p)$ when $p$ is split. 
\end{itemize}
We will denote by $K_p(p^k)$ the compact-open subgroup of matrices congruent to $1$ modulo $p^k$, and by $\LG_p$ the Lie algebra of $K_p$. 

The crucial case is when we have $K_p'=K_p$ for all but one rational prime $p$. We identify $\N(F_p)$ with $F_p$ (see \ref{unipotent} above) and for a finite-index subgroup $L$ of $\N(F_p)$ we put  
$$
\alpha_1(L) = \min\{|v|_p^{-1} \: :\: v\in L, v\not=0 \}
$$
and for any $v_1\in L$ such that $|v_1|_p=\alpha_1(\Lambda)$
$$
\alpha_2(L) = \min\{|v|_p^{-1} \: :\: v\in L, v\not\in \ZZ v_1 \}.
$$
If $g\in K_p$ and $\Lambda = g\N(F)g^{-1}\cap K_p',\, L=gN_pg^{-1}\cap K_p'$ (where $N_p=K_p\cap\N(F_p)$) then we have 
$$
\alpha_i(\Lambda) \asymp \alpha_i(L), \, i = 1, 2
$$
with absolute constants, so that we must bound the sum
\begin{equation} \label{unip_contrib_local}
S_p = \sum_{g \in N_p \bs K_p / K_p'} \frac{\alpha_2(g^{-1}N_p g \cap K_p')} {\alpha_1(g^{-1}N_p g \cap K_p')} . 
\end{equation}
We rewrite $S_p$ as follows: we fix a $k\ge 1$ such that $K_p(p^k)\subset K_p'$. Then the quantities $\alpha_i(g^{-1}N_pg\cap K_p')$ are constant on a $K_p'$-orbit in $K_p/N_pK_p(p^k)$; on the other hand the cardinality of the $K_p'$-orbit of $gN_pK_p(p^k)$ in $K_p/N_pK_p(p^k)$ is equal to $\frac{|K_p'/K_p(p^k)|}{|(g^{-1}N_p g \cap K_p')K_p(p^k)/K_p(p^k)|}$ so that 
\[
S_p = \sum_{g\in K_p / N_p K_p(p^k)} \frac{|(gN_p g^{-1} \cap K_p')K_p(p^k) / K_p(p^k)|}{[K_p':K_p(p^k)]}\times\frac{\alpha_2(gN_pg^{-1}\cap K_p')}{\alpha_1(gN_pg^{-1}\cap K_p')}
\]
From the equality
\[
\alpha_2(gN_pg^{-1}\cap K_p') \alpha_1(g^{-1}N_pg\cap K_p') = |gN_pg^{-1} / (gN_pg^{-1} \cap K_p')|
\]
it follows that $\alpha_1\alpha_2 = \frac{|N_p / (N_p \cap K_p(p^k))|}{|g^{-1}N_p g \cap K_p'|}$ and then that: 
\begin{align*}
S_p &= \frac{[N_p : N_p\cap K_p(p^k)]}{[K_p':K_p(p^k)]} \sum_{g\in K_p/N_pK_p(p^k)} \frac 1 {\alpha_1(gN_pg^{-1}\cap K_p')^2} \\
&= [K_p : K_p'] \cdot \frac{[N_p : N_p\cap K_p(p^k)]}{[K_p : K_p(p^k)]} \sum_{g\in K_p/N_pK_p(p^k)} \frac 1 {\alpha_1(gN_pg^{-1}\cap K_p')^2}. 
\end{align*}
On the other hand $B_p$ normalizes $N_p$ and we can mod out on the right to get:
\begin{equation} \label{Sp'}
S_p  = [K_p:K_p'] \cdot \frac{|B_p : (B_p \cap K_p(p^k))]}{[K_p : K_p(p^k)]} \sum_{g\in K_p/B_pK_p(p^k)} \frac 1 {\alpha_1(gN_pg^{-1}\cap K_p')^2}
\end{equation}
which is the sum that we will now estimate. 

For $l=0,\ldots,k-1$ we define
\begin{equation} \label{def_Xl}
X_l = \{g\in K_p/B_pK_p(p^k) \: :\: gN_pg^{-1}\cap K_p'\subset K_p(p^l),\, \not\subset K_p(p^{l+1}) \} 
\end{equation}
and put $d_l=|X_l|$. Then for $g \in X_l$ we have $\alpha_1(gN_pg^{-1}\cap K_p') = p^l$ and so: 
$$
S_p \le \frac{[K_p:K_p']}{2p^{2k}} \: \sum_{l=0}^{k-1} d_l p^{-2l} 
$$
We may suppose that $K_p(p^{k-1})\not\subset K_p'$, and we will prove in 
\ref{pfdl} below the following estimate for $d_l$ when $l\le k/3$:
\begin{equation} \label{dl}
d_l \le p^{\frac{17k}9}.
\end{equation} 
In general we have trivially that $d_l \le |K_p/B_pK_p(p^k)| \ll p^{2k}$. It follows that 
$$
S_p \ll \frac{[K_p:K_p']}{p^{2k}} \sum_{l=0}^{\lfloor k/3\rfloor} p^{\frac{17k}9} + [K_p:K_p'] \sum_{l=\lfloor k/3\rfloor+1}^{k-1} p^{-2l} \le k\frac{[K_p:K_p']}{p^{k/9}} + 2\frac{[K_p:K_p']}{p^{2k/3}}
$$
on the other hand we can estimate trivially $[K_p:K_p'] \le 2p^{6k}$ and $k\ll p^{\eps k}$ for any $\eps>0$, uniformly in $k$ and $p$, so we finally get 
\begin{equation}
\sum_{j=1}^{h_{K'}} \frac {\alpha_2(\Lambda_{K',j})} {\alpha_1(\Lambda_{K',j})}  \ll [K_p:K_p']^{1-\frac 1{55}}.
\label{pScusp}
\end{equation}

Now we return to the general case ; let $m$ be an integer such that $K_f(m)\subset K_f'$, as above we have that 
$$
S = \frac{[K_f:K_f']}{[K_f:N_fK_f(m)]} \sum_{g\in K_f/N_fK_f(m)} \frac 1{\alpha_1(gN_fg^{-1}\cap K_f')^2}.
$$
Let $N_f=\N(\Ade)\cap K_f$. For any prime $p$ dividing $m$, $gN_pg^{-1}\cap K_p'$ is  the pro-$p$ summand of $gN_fg^{-1}\cap K_f'$, so that we have
\[
gN_f g^{-1} \cap K_f' = \prod_{p|m} gN_p g^{-1} \cap K_p'
\]
and it follows that
\[
\alpha_1(gN_fg^{-1}\cap K_f') = \prod_p \alpha_1(gN_p g^{-1} \cap K_p').
\] 
So we get, writing $m = \prod_p p^{k_p}$: 
$$
S = \frac{[K_f:K_f']}{\prod_p [K_p:N_pK_p(p^k_p)]} \prod_p \left(\sum_{g\in K_p/N_pK_(p^k_p)} \frac 1{\alpha_1(gN_pg^{-1}\cap K_p')^2} \right).
$$
We can rewrite this as 
$$
S = \frac{[K_f:K_f']}{\prod_p [K_p:K_p']} \prod_p S_p \ll \frac{[K_f:K_f']}{\prod_p [K_p:K_p']^{\frac 1{55}}} 
$$
where the second inequality follows from \eqref{pScusp}. It follows from \cite[Lemma 5.11]{7S} that there are constants $c,C\ge 1$ such that for any compact-open subgroup $K_f'\subset K_f$, if $K_p'$ is its projection to $K_p$ then we have 
$$
[K_f:K_f'] \le C \left( \prod_p [K_p:K_p'] \right)^c
$$
where the product runs over all rational primes such that $K_p'\not=K_p$, so that we get
$$
S \ll [K_f:K_f']^{1-\frac 1{55c}}
$$
which finishes the proof of \eqref{Scusp} (we get $\delta=\frac 1{55c}$).


\subsubsection{Proof of \eqref{dl}}
\label{pfdl}
This proof is reminescent of that of \cite[Proposition 5.13]{7S}, albeit much more cumbersome due to the fact that we cannot identify the precise elements of $N_p$ which are conjugated into $K_p'$. Under the hypothesis that $K_p'\not\supset K_p(p^{k-1})$, for any $l=1,\ldots, k-1$ we have that $K_p'K_p(p^{l+1})/K_p(p^{l+1})$ cannot contain a generating set for the 6-dimensional $\FF_p$-Lie algebra 
$\LG_p/p\LG_p=K_p(p^l)/K_p(p^{l+1})$. For a subset $Y\subset K_p$ define 
$$
q_Y(j) = \max_{h\in K_p/K_p(p^j)} |(hK_p(p^j)\cap Y)B_pK_p(p^{j+1})/B_pK_p(p^{j+1})|. 
$$
Then we have : 
\begin{equation} \label{kkkb}
|X_l|\le \prod_{j=0}^{k-1} q_{X_l}(j). 
\end{equation}
We will prove the following lemma at the end of the section (recall that $\LG_p$ is the $\ZZ_p$-Lie algebra associated to $K_p$). 

\begin{lemma}
If $p$ is unramified and $p\not=2,3$ then 
a proper subgroup of $K_p/K_p(p)$ (resp. a proper Lie subalgebra of 
$\LG_p/p\LG_p$) cannot contain more than $p+1$ pairwise noncommuting unipotent 
(resp. nilpotent) elements. 
\label{brzouk}
\end{lemma}

This implies that $q_{X_l}(0)\le p+1$ for all $l\le k/3$ (since there are 
only finitely many ramified primes we get $q_{X_l}(0)\le Cp$ for a $C>0$ depending 
on $F$, we will work with $C=1$ to simplify notation). Now we deal 
with $j\ge 1$: we will prove that when $j < (k-2l)/3$ we must have 
$q_{X_l}(j)\le p$, which  in view of \eqref{kkkb} implies immediately \eqref{dl} 
for $l\le k/3$. Suppose that there is an $h\in K_p/K_p(p^j)$ such that 
$$
|(hK_p(p^j)\cap Y)B_pK_p(p^{j+1})/B_pK_p(p^{j+1})| > p ; 
$$
conjugating $K_p'$ by $h$ we may suppose that $h=1$. This means that there 
exists pairwise distinct $c_i\in\so_p-p\so_p,i=0,\ldots,p$ such that for each 
$i$ there is a $t_i\in \so_p-p\so_p$ satisfying
$$
\left(1 + \begin{pmatrix} *&*\\p^jc_i&*\end{pmatrix} \right) \begin{pmatrix}1&p^l t_i\\ &1\end{pmatrix} \left(1 - \begin{pmatrix} *&*\\p^jc_i&*\end{pmatrix}\right) \in K_p'. 
$$
Computing the right-hand side yields that
\begin{equation}
g_i = 1+\begin{pmatrix} p^{l+j}t_ic_i & p^lt_i \\ p^{l+2j}t_ic_i^2 & - p^{l+j}t_ic_i\end{pmatrix}\in K_p'.
\label{belong}
\end{equation}
Now the worst that can happen is that we are (up to conjugation) in at most 
one of the following situations:
\begin{itemize}
\item[(a)] All $t_i$ are in $\ZZ_p$;
\item[(b)] All $t_ic_i$ are in $\ZZ_p$;
\item[(c)] All $t_ic_i^2$ are in $\ZZ_p$. 
\end{itemize} 
In case (a) we get that $K_p'K_p(p^{l+2j+1})$ contains the 
subgroup $1+p^{l+2j}V$ where 
$$
V=\left\{ \begin{pmatrix} x&y\\z&-x\end{pmatrix} \: :\:  x,z\in\so_p/p\so_p,\, y\in\FF_p \right\}.
$$
We may suppose that all $t_m=1$; now let $i,i'$ such that 
$a=c_i+c_{i'}\not\in\FF_p$, modulo $p^{2l+j+1}$ we have
$$
g_ig_{i'}^{-1} = 1 + p^{2l+j} \begin{pmatrix} 0&a\\0&0\end{pmatrix}
$$
which is not in $1+p^{2l+j}V$ so that we see that $K_p'$ contains 
$K_p(p^{2(l+j)})$. In case (c) we can do exactly the same reasoning to 
get that $K_p'\supset K_p(p^{2l+3j})$. It remains to deal with case (b), 
which is again similar: we have that $K_p'K_p(p^{l+2j+1})$ contains $1+p^{l+2j}V'$,
$$
V'=\left\{ \begin{pmatrix} x&y\\z&-x\end{pmatrix} \: :\:  y,z\in\so_p/P\so_p,\, x\in\FF_p \right\}.
$$
and if we suppose that all $t_mc_m=1$ and $t_i+t_{i'}\not\in\FF_p$ we get 
that 
$$
g_ig_{i'}^{-1} = 1 + p^{2l} \begin{pmatrix} p^j(t_i+t_{i'}) & t_i-t_{i'}\\ 0 & p^j(t_i+t_{i'}) \end{pmatrix} 
$$
modulo $p^{2l+j}$, 
and multiplying by some other $g_m$s to kill the top-right coefficients we 
get that $1+p^{2l+j}u\in K_p'K_p(p^{2l+j+1})$ for some $u\not\in V'$, which shows 
that $K_p'$ contains $K_p(p^{2(l+j)})$ also in this case. 

In conclusion, we have seen that if $q_{X_l}(j)>p$ then $K_p'$ contains 
$K_p(p^{2l+3j})$ which implies that $j\ge(k-2l)/3$, which finishes the proof 
of \eqref{dl}.


\subsubsection{Proof of Lemma \ref{brzouk}}
It follows from the following classification as the image $X$ of any of the 
proper subgroups listed here contains less than $p+1$ unipotent, pairwise 
noncommuting elements.

\begin{lemma}
Let $H$ be a subgroup of $K_p/K_p(p)$ such that $H$ contains two unipotent 
elements which do not commute ; then
\begin{itemize}
\item If $p$ is inert then either $H=\SL_2(\FF_{p^2})$ or $H$ is conjugated to 
$\SL_2(\FF_p)$; 
\item If $p$ is split then either $H=\SL_2(\FF_p)\times\SL_2(\FF_p)$ or 
$H=\phi(\SL_2(\FF_p))$ where $\phi=(\phi_1,\phi_2)$ for some endomorphisms 
$\phi_1,\phi_2$ of $\SL_2(\FF_p)$. 
\end{itemize}
There is a similar statement for proper Lie subalgebras of $\LG_p$. 
\end{lemma} 

In case $p\not=3$ is inert this follows immediately from Dickson's Theorem 
\cite[Theorem 8.4 in Chapter 2]{Gorenstein} (this is where we use $p\not=3$). 
In the remaining case where $p$ is split the lemma is actually much simpler, 
since the projection of $H$ on one of the factors $\SL_2(\FF_p)$ must 
contain two non-commuting unipotent elements and we can then apply 
Dickson's theorem (which in this case is almost trivial). 

The result for Lie algebras is easier in the inert case: if a subalgebra 
contains two 
noncommuting nilpotent elements then their Lie bracket is an element in 
the Cartan subalgebra contained in the intersection of their normalizers. 
If this bracket is not $\FF_p$-rational then we get the whole algebra since 
its adjoint action on each of the nilpotent $\FF_{p^2}$-subalgebras is 
irreducible (here we use $p\not=2$), if it is then they generate a subalgebra 
conjugated to $\LSL_2(\FF_p)$. 
The ramified case is dealt with as for the case of groups.


\subsection{Limit multiplicities}

Let $\rho,V$ be a finite-dimensional representation of $\SL_2(\CC)$ and for an hyperbolic orbifold $M$ let $\Delta^p[M]$ be the Hodge Laplacian on $L^2\Omega^p(M;V)$, and $m_V^p(\lambda;M)=\dim\ker(\Delta^p[M]-\lambda)$. There are also $L^2$-spectral measures $\nu^p$, which are Borel measures on $[0,+\infty[$ obtained by pushing forward the Plancherel measure. We will prove in \ref{repr} below that Theorem \ref{multlim} follows from the following less precise result. 

\begin{theo}
For any regular Borel set $S\subset[0,+\infty[$ and $p=0,\ldots,3$ we have 
$$
\lim_{n\to+\infty}\frac{\sum_{\lambda\in S} m_V^p(\lambda;M_n)}{\vol M_n} = \nu^p(S)
$$
where $M_n$ is as in Theorem \ref{multlim}. 
\label{multlimvp}
\end{theo}

\subsubsection{Regularized trace}
The first step towards Theorem \ref{multlimvp} is to prove the convergence of regularized traces; the following result is an immediate consequence of Theorem 4.5 in \cite{volI} and Theorem \ref{BSbianchi} above. 

\begin{prop}
Let $\Gamma_n$ be a sequence of cusp-uniform congruence subgroups of a 
given Bianchi group. Then we have the limit 
$$
\lim_{n\to+\infty}\frac{\otr_R\phi(\Delta^p[M_n])}{\vol M_n} = \otr^{(2)}\phi(\Delta^p[\HH^3])
$$
for any $\phi\in\mathcal{A}(\RR)$. 
\label{cvtr}
\end{prop}

Recall from {\it loc. cit.} that $\mathcal{A}(\RR)$ is a $C^\infty$-dense 
subset of the Schwartz functions on $\RR$ whose Fourier transforms 
yield point-pair invariants of rapid decay on $\HH^3\times\HH^3$. 

\subsubsection{Proof of Theorem \ref{multlimvp}}
Because $\mathcal{A}(\RR)$ is dense in $L^2(\RR)$ it suffices (by approximating the characteristic function of regular sets) to prove that for any $\phi\in\mathcal{A}(\RR)$ we have that $\otr\phi(\Delta_\cusp^p[M_n])/\vol M_n$ converges to $\otr^{(2)}\phi(\Delta^p[\HH^3])$. Let $\Delta_\cusp^p[M_n]$ be the restriction of $\Delta^p[M_n]$ to the subspace of cusp forms. By standard arguments one deduces Theorem \ref{multlimvp} from the fact that for all $\phi$ in a $L^1$-dense subset of $C_c^\infty(\RR)$ we have 
\begin{equation}
\frac{\otr\phi(\Delta_\cusp^p[M_n])}{\vol M_n} = \otr^{(2)}\phi(\Delta^p[\HH^3])
\label{cvtrcusp}
\end{equation}
According to Proposition \ref{cvtr} this would follow if we can prove 
$$
\otr_R \phi(\Delta^p[M_n]) - \otr \phi(\Delta_\cusp^p[M_n]) = o(\vol M_n).
$$ 
We will show this when $p=0$. From \eqref{spec} we get:
\begin{align*}
&\otr_R \phi(\Delta^0[M_n]) - \otr \phi(\Delta_\cusp^0[M_n])  \\
&\qquad = 0/1 + \frac 1 4 \sum_{l=-2q}^{2q}d_l \phi\left(\left(1-\frac{|l|}2 \right)^2+\lambda_V \right) \tr\Psi_l(0) \\
&\qquad \quad - \frac 1{2\pi} \int_{-\infty}^{+\infty} \sum_{l=-2q}^{2q} d_l \phi\left(\left(1-\frac{|l|}2\right)^2+u^2+\lambda_V\right)\tr\left(\Psi_l(iu)^{-1}\frac{d\Psi_l(iu)}{du}\right) du 
\end{align*}
(where the summand 0/1 comes from the subspace of harmonic sections, equals 0 when $V$ is acyclic). As $\Psi(iu)$ is unitary for $u\in\RR$ we obtain
$$
\frac 1 4 \sum_{l=-2q}^{2q}d_l \phi\left(\left(1-\frac{|l|}2 \right)^2+\lambda_V \right) \tr\Psi_l(0) \le C\sum_{l=-2q}^{2q} d_l \le Ch_n. 
$$
Putting 
$$
\xi(u) = \max_{l=-2q,\ldots,2q} \phi\left(\left(1-\frac{|l|}2\right)^2+u^2+\lambda_V\right)
$$ 
and applying Proposition \ref{estenta} we get for any $\eps>0$ the bound
\begin{equation*}
|\otr_R K_{\phi,0}^{\Gamma_n} - \otr (K_{\phi,0}^{\Gamma_n})_\disc| \le  Ch_n + \int_{-\infty}^{+\infty} \xi(u) C_\eps(u)  du \, h_n [\Gamma:\Gamma_n]^\eps
\end{equation*} 
where the integral on the right-hand side converges absolutely since $C_\eps(u)$ is polynomially bounded. By Lemma \ref{nbcusps} we get that for $\eps$ small enough it is in fact an $o(\vol M_n)$, which finishes the proof of \eqref{cvtrcusp} and of the theorem. 


\subsubsection{Laplacian eigenvalues and representations}\label{repr}

The fact that we can deduce limit multiplicities for representations (Theorem \ref{multlim}) from limit multiplicities for Laplacian eigenvalues (Theorem \ref{multlimvp}) is a consequence of the fact, which is specific to real-rank-one groups, that a unitary representation of $\SL_2(\CC)$ is determined by its Casimir eigenvalue and its $\SU(2)$-types. More precisely, the unitary representations of $\SL_2(\CC)$ are parametrized by $(\ZZ\times i\RR)/\sim\cup]0,2[$ where $(l,ia)\sim(-l,-ia)$. On the other hand  the $\SU(2)$-types are the restrictions to $\SU(2)$ of the holomorphic representations $V_n=V_{n,0}$ of $\SL_2(\CC)$, and by Frobenius reciprocity the representations containing the $\SU(2)$-type $V_n$ are the $\pi_{\pm l,ia}$ for $0\le l\le n$ and $n-l=0\pmod 2$, so we can deduce Theorem \ref{multlim} by an easy induction on $l$ using the limit multiplicities for $m_{V_l}^0$ or $m_{V_l}^1$.

\subsection{Approximation for analytic torsion}

\begin{theo}
Let $\Gamma_n$ be a cusp-uniform sequence of torsion-free congruence subgroups of $\Gamma$ $M_n=\Gamma_n\bs\HH^3$, then we have 
$$
\lim_{n\to\infty} \frac{\log T_R(M_n;V)}{\vol M_n} = t^{(2)}(V). 
$$
\label{analytic}
\end{theo}

\begin{proof}
According to Theorem A in \cite{volI} we have to check two conditions:
\begin{itemize}
\item The sequence $M_n$ is BS-convergent to $\HH^3$
\item There is an $\eps>0$ such that there exists a $C>0$ so that for all $u\in[-\eps,\eps]$ we have 
\begin{equation}
\tr\left( \Psi_l(iu)^{-1}\frac{d\Psi_l(iu)}{du} \right) , \: \tr\left( \Phi_l(iu)^{-1}\frac{d\Phi_l(iu)}{du} \right) = o(\vol M_n).
\label{intn0}
\end{equation}
\end{itemize}
The BS-convergence is the content of Theorem \ref{BSbianchi} above. To prove \eqref{intn0} note that we have for all $u\in\RR$ the bound
\begin{equation}
\tr\left( \Psi_l(iu)^{-1}\frac{d\Psi_l(iu)}{du} \right) \ll \left| \frac{d\Psi_l(iu)}{du} \right| h_n d
\label{sst}
\end{equation}
where $d=\dim V$ and $h_n$ is the number of cusps of $M_n$, since $\Psi(s)$ operates on a vector space of dimension $h_nd$ and it is unitary for $\Real(s)=1/2$. Now we have $h_n\le (\vol M_n)^{1-\delta}$ for some $\delta>0$, according to Lemma \ref{nbcusps}, and on the other hand according to Proposition \ref{estenta}\footnote{The operators $\Psi_l(s)$ are intertwined with the $K_f'$-invariant, $\chi_l,\rho$-isotypic matrix block of $\Psi(s)$ according to \eqref{eis}.} for all $\eps>0$ there exists $C_\eps(u)$ such that 
$$
\left| \frac{d\Psi_l(iu)}{du} \right| \le C_\eps(u)|\fri_n|^\eps \ll C_\eps(u)(\vol M_n)^{3\eps} .  
$$
(where the second majoration follows from Lemma \ref{index-level}), and taking $\eps=\delta/6$ shows that the right-hand side of \eqref{sst} is indeed an $o(\vol M_n)$, uniformly for $u$ in a compact set. 
\end{proof}

%% file: entrelacement.tex
This section is devoted to the proof of the following result (see \ref{spec_dec} for notations). 

\begin{prop}
Let $\tau_\infty$ be a finite-dimensional representation of $K_\infty$, $\fri$ an ideal of $\so_F$ and $\chi$ a Hecke character such that $\frf_\chi | \fri$. Let $\tau_f$ be the representation of $K_f$ on $\CC[K_f/K_f(\fri)]$, $\tau=\tau_\infty\otimes\tau_f$ and $\phi\in\sh(\chi,\tau)$. Let $\eps>0$; there exists a polynomially bounded function $C_\eps$ on $\RR$ depending only on $F,\eps$ and $\tau_\infty$ such that 
\begin{equation}
\left\|\frac d{du}\Psi\left(\frac 1 2 +iu\right)\phi \right\|_\sh \le C_\eps(u) |\frf_\chi|^\eps \|\phi\|_\sh.
\end{equation}
for all $u\in\RR$. 
\label{estenta}
\end{prop}

We will suppose that $\phi\in C^\infty(K)$ is equal to a product $\otimes_v\phi_v$; then it is enough to show that 
\begin{equation*}
\left\|\frac d{du}\Psi\left(\frac 1 2 +iu\right)\phi \right\|_{L^2(K)} \ll |\frf_\chi|^\eps 
\end{equation*}
because $\|.\|_{\sh}=\sqrt{h_F}\|.\|_{L^2(K)}$ for $A^1$-equivariant functions. Since $\sh_s(\chi)$ is an irreducible unitary representation of $\SL_2(\Ade)$ when $\Real(s)=1/2$ and $\Psi(s)^{-1}\frac d{du}\Psi(s)$ is a $\SL_2(\Ade)$-equivariant endomorphism of $\sh_s(\chi)$ it is a scalar operator, say $c\id$. Now $\Psi(s)^{-1}$ is unitary and thus for any two $\phi,\phi'\in\sh_s(\chi)$ we have 
$$
\frac{ \left| \frac d{du} \Psi(s)\phi\right|_\sh}{|\phi|_\sh} = \frac{\left|\frac d{du}\Psi(s)\phi'\right|_\sh}{|\phi'|_\sh}
$$
so that it suffices to consider a single function $\phi\in\sh_s(\chi)$; we will take $\phi_v$ to be the spherical vector at unramified places and specify $\phi_v$ for each ramified place; the infinite place does not matter very much for our purposes here.


\subsection{Computation of the intertwining integrals}

Let $v$ be a finite place; we will make repeated use of the matrix decomposition:
\begin{equation}
w\begin{pmatrix}1&x\\ &1\end{pmatrix}=\begin{pmatrix}x^{-1}&-1\\ & x\end{pmatrix}\begin{pmatrix}1& \\x^{-1}&1\end{pmatrix}\in\B(F_v) K_v \text{ when } x\in F_v-\mathcal{O}_v.
\label{Iwasawa}
\end{equation}

\subsubsection{Ramified places}
\label{ramify}
We suppose that $v\in S_\chi$, that is $\chi$ is non-trivial on $\so_v^\times$; we suppose that $\chi_v$ is trivial on $1+\pi_v^{m_v}\so_v$. From \eqref{Iwasawa} and \eqref{prolongement} it follows that:
\begin{align*}
\Psi_v(s)\phi_v &= 
           \int_{|x|_v>1}|x|^{-2s} \chi^{-1}(x) \phi_v\left(\begin{pmatrix}1&\\x&1\end{pmatrix} k\right) dx
           + \int_{\mathcal{O}_v}\phi_v\left(\begin{pmatrix}1&\\x&1\end{pmatrix} w^{-1}k\right)dx \\
        &=: I_v(k) + J_v(k)
\end{align*}
We now compute the $L^2$-norms of $I_v$ and $J_v$ for the function $\phi_v$ defined as follows:
\begin{equation*}
\phi_v(k)= \begin{cases}
           \chi(m) & \text{if } k=mn\in B_v ; \\
           0       & \text{otherwise.}
         \end{cases}
\end{equation*}
One can then compute that for $k=w\begin{pmatrix}a&b\\c&d\end{pmatrix}$ we have:
\begin{align*}
J_v(k) &= \frac 1{q_v^{m_v}} \sum_{x\in\so_v/\pi_v^{m_v}\so_v} \phi_v\left(\begin{pmatrix} a&b\\ax+c & bx+d\end{pmatrix}\right) = \frac 1 {q_v^{m_v}}\sum_{x,ax+c=0}\chi(a) \\
       &= \begin{cases}
             \frac 1 {q_v^{m_v}}\chi(a) & \text{if } a\in\so_v^\times; \\
             0                        & \text{otherwise}
          \end{cases}
\end{align*}
and it follows immediately that
\begin{equation}
|J_v|_{L^2(K_v)}^2 = \frac 1{q_v^{2m_v-1}(q_v+1)} = q_v^{-m_v}|\phi_v|_{L^2(K_v)}^2
\label{Jv}
\end{equation}

Since $\int_{\so_v^\times}\chi(x)dx=0$ and our function $\phi_v$ is $K_v(\pi_v^{m_v})$-left-invariant we have 
$$
\int_{|x|_v\ge m_v}|x|^{-2s} \chi^{-1}(x) \phi_v\left(\begin{pmatrix}1&\\x&1\end{pmatrix} k\right) dx = 0
$$
and it follows that
\begin{align*}
I_v(k) &= \sum_{l=1}^{m_v-1} q_v^{-2s} \int_{|x_v|=l} \chi(x)^{-1} \phi_v\left(\begin{pmatrix}1&\\ x^{-1}&1\end{pmatrix} k\right) dx \\
       &= \sum_{l=1}^{m_v-1} q_v^{l-2s}\chi(\pi_v)^{-l} \int_{\so_v^\times} \chi(x)^{-1} \phi_v\left(\begin{pmatrix}1&\\ \pi_v^lx&1\end{pmatrix} k\right) dx 
\end{align*}
For $k=\begin{pmatrix}a&b\\c&d\end{pmatrix}$, $c\in\so_v^\times$ the right-hand side equals 0. The summands are also zero if $c\in\pi_v^m\so_v^\times$ for some $m=1,\ldots,m_v-1$: indeed, the sum restricts to the summand $l=m$ and we get
\begin{align*}
I_v(k) &= q_v^{l-2s}\chi(\pi_v)^{-m} \int_{\so_v^\times} \chi(x)^{-1} \phi_v\left(\begin{pmatrix}a & b\\ \pi_v^max + c &\pi_v^mbx + d \end{pmatrix} \right) dx \\
       &= q_v^{l-2s}\chi(\pi_v)^{-m} \chi(a) \int_{(-\pi_v^{-m}ca^{-1})(1+\pi_v^{m_v-m}\so_v)} \chi(x) dx \\
       &= q_v^{l-2s}\chi(\pi_v)^{-2m} \chi(c) \int_{1+\pi_v^{m_v-m}\so_v} \chi(x) dx
\end{align*}
and since $\chi$ is nontrivial on $1+\pi_v^{m_v-m}\so_v$ the integral on the 
right vanishes. 

\subsubsection{Unramified places}
\label{L}
If $\chi$ is not ramified at $v$ then we choose $\phi_v$ to be the function in $\sh_s$ which is identically equal to 1 on $K_v$. We recall the following well-knwown computation:
\begin{align*}
\Psi_v(s)\phi_v(k) &= \int_{F_v-\mathcal{O}_v}\phi_s\left(w\begin{pmatrix}1&x\\ &1\end{pmatrix}k\right)dx + 1 \\
                   &= \int_{F_v-\mathcal{O}_v} |x|_v^{-2s}\chi(x)^{-1}\phi\left(\begin{pmatrix}1& \\x^{-1}&1\end{pmatrix} k \right)dx + 1\\
                   &= 1 + \sum_{l\ge 1}(1-q_v^{-1})q_v^l\chi(\pi_v)^k q_v^{-2sl} = \frac{1-\chi(\pi_v)q_v^{-2s}}{1-\chi(\pi_v)q_v^{-2s+1}}.
\end{align*}

\subsubsection{Infinite place}
\label{Iinf}
We finally compute
$$
I_\infty(k) = \Psi_\infty(s)\phi_\infty(k).
$$ 
For $z\in\CC$ we have 
\begin{equation*}
w\begin{pmatrix}1&z\\ &1\end{pmatrix} = \begin{pmatrix}u^{-1} & u^{-1}\bar{z}\\ & u\end{pmatrix} k_z, \: u=\sqrt{|z|_\infty +1},\: k_z=\begin{pmatrix}-u^{-1}\bar{z} & u^{-1} \\ u^{-1} & u^{-1}z \end{pmatrix}
\end{equation*}
and it follows that 
\begin{equation*}
I_\infty(k) = \int_\CC(|z|_\infty+1)^{-2s}\phi_\infty(k_zk) dzd\bar{z}
\end{equation*}
As $|z|\to+\infty$ we have $\left|k_z-\begin{pmatrix}\bar{z}/|z| &\\ &z/|z|\end{pmatrix}\right|\ll |z|_\infty^{\frac{-1}2}$ (for any norm $|.|$ on $K_\infty$). For $\Real(s)>1/2$ we get:
\begin{equation}
I_\infty = \phi_\infty(k) \int_\CC \chi_\infty(z)(|z|_\infty+1)^{-2s}dzd\bar{z} + O(1)
\label{infini}
\end{equation}
where the $O(1)$ depends on $\tau_\infty$ but not on $s$. If $\chi_\infty\not=1$ then the integral is zero so that $I_\infty$ is bounded independantly of $s$ for $\Real(s)\ge 1/2$. If $\chi_\infty=1$ we have 
\begin{equation*}
\int_{\mathbb{C}}(|z|_\infty+1)^{-2s}dzd\bar{z} = \pi\frac{\Gamma(2s-1)}{\Gamma(2s)}
\end{equation*}
and it follows that $I_\infty$ has a meromorphic continuation to $\Real s>0$ such that $I_\infty(k)-\pi\frac{\Gamma(2s-1)}{\Gamma(2s)}\phi_\infty(k)$ is bounded independantly of $s$ for $\Real(s)\ge 1/2$.

\subsubsection{Final expression} 
For $\Real(s)\ge 1/2$ and our specific $\phi$ we get the formula
\begin{equation} \label{calcul!}
\Psi(s)\phi(k) = I_\infty \times \frac{L(\chi,2s-1)}{L(\chi,2s)} \prod_{v\in S_\chi} J_v(k)
\end{equation}


\subsection{Proof of Proposition \ref{estenta}}

We will write $s=\sigma+iu$ for this whole subsection and suppose (unless otherwise stated) that $\sigma=1/2$. Taking the derivative of the product \eqref{calcul!} yields
$$
\frac d{du} \Psi(s)\phi(k) = \left(\frac{\frac d{du}L(\chi,2s)}{L(\chi,2s)} + \frac{\frac d{du}\left(L(\chi,2s-1)I_\infty(k)\right)}{L(\chi,2s-1)I_\infty(k)} \right) \Psi(s)\phi(k)
$$
so that we get, using the functional equation for $L(\chi,.)$:
\begin{align*}    
\left| \frac d{du} \Psi(s)\phi \right|_{L^2(K)} &\le 2\left|\frac{\frac d{du}L(\chi,2s)}{L(\chi,2s)}\right| + \frac{\left|\frac d{du}(\gamma(s)I_\infty)\right|_{L^2(K_\infty)}}{|\gamma(s) I_\infty|_{L^2(K_\infty)}} \\
   &=: L_1 + L_2.
\end{align*}
We will suppose at first that $\chi$ is non-trivial so that the $L$-function $L(\chi,.)$ is holomorphic on $\Real(s)>0$. To bound both $L_1$ and $L_2$ we use the following well-known lemma.  

\begin{lemma}
Suppose that $\chi$ is non-trivial ; then we have 
\begin{equation}
\left|\frac{\frac d{du}L(\chi,2s)}{L(\chi,2s)}\right| \ll (\log|\frf_\chi|)^2
\label{dvlogL}
\end{equation}
with a constant depending only on $s$ and $F$, growing polynomially in $\Ima(s)$.
\end{lemma}

\begin{proof}
The Euler product for $L(\chi,2s)$ yields for $\Real(s)>1/2$ the absolutely converging series expansion 
\begin{equation*}
\frac{\frac d{du}L(\chi,2s)}{L(\chi,2s)} = \sum_{v\not\in S_\chi} \frac{2i(\log q_v) q_v^{-2s} \chi(\pi_v)}{1-\chi(\pi_v)q_v^{-2s}}.
\end{equation*}
We get 
\begin{equation*}
\frac{\frac d{du}L(\chi,2s)}{L(\chi,2s)} = \sum_{v\not\in S_\chi} 2i(\log q_v)q_v^{-2s}\chi(\pi_v) \sum_{k\ge 0}\chi(\pi_v)^kq_v^{-2ks}.
\end{equation*}
The series $\sum_{v\not\in S_\chi} (\log q_v)q_v^{-2s}\chi(\pi_v) \sum_{k\ge 1}\chi(\pi_v)^kq_v^{-2ks}$ converges absolutely for $\Real(s)>1/4$ and its sum is bounded by a constant depending only on $F$, so that we are left with estimating $\sum_{v\not\in S_\chi} (\log q_v)q_v^{-2s}\chi(\pi_v)$. 

Let $\chi_1,\ldots,\chi_{h_F}$ be all the Hecke characters on $F^\times\bs \Ade^1$ such that $\ker(\chi_j)\supset M$. If $\chi$ is any Hecke character there exists some $j$ such that $\chi(\pi_v)=\chi_j(\pi_v)$ for all places $v\not\in S_\chi$. We then have that  
$$
\sum_{v\not\in S_\chi} q_v^{-2s}\log(q_v)\chi(\pi_v) = \sum_{v\not\in S_{\chi_j}} q_v^{-2s}\log(q_v)\chi_j(\pi_v) - \sum_{v\in S_\chi-S_{\chi_j}} q_v^{-2s}\log(q_v)\chi_j(\pi_v). 
$$
As $\chi_j$ is non-trivial the function $H : s\mapsto \sum_{v\not\in S_{\chi_j}} q_v^{-2s}\log(q_v)\chi_j(\pi_v)$ has an holomorphic extension to an open subset of $\CC$ containing the half-plane $\Real(s)\ge 1/2$, and by standard arguments\footnote{It suffices to prove that there is a polynomial bound for the logarithmic derivative of a given Hecke $L$-function} there is a polynomial bound (depending only on $F$ as $\chi_1,\ldots,\chi_{h_F}$ are fixed) in $u$ for $H(1/2+iu)$. On the other hand, putting $q=\max_{v\in S_\chi} q_v$ we get that for $\Real(s)=1/2$ we have 
$$
\left|\sum_{v\in S-S_\chi} (\log q_v)q_v^{-2s}\chi(\pi_v)\right| \le \log q \sum_{v\in S-S_\chi} \frac 1{q_v} \ll (\log q)^2
$$
and as $q\le|\frf_\chi|$ we are left with  
$$
\left|\frac{\frac d{du}L(\chi,2s)}{L(\chi,2s)}\right| \le C(u)(\log|\frf_\chi|)^2
$$
where $C(u)$ has polynomial growth in $u$, which finishes the proof of the lemma. 
\end{proof}

Since $\frac d{du}I_\infty\gamma(s)$ is also bounded by 
$\log|\frf_\chi|$ we thus get 
\begin{equation}
|L_1|, |L_2| \le \left|\frac{\frac d{du}L(\chi,2s)}{L(\chi,2s)}\right| \le C(u) (\log|\frf_\chi|)^2
\label{L12}
\end{equation}
where $C(u)$ is growing polynomially. 

It remains to deal with the case where $\chi=1$, which is the same except that we have to group the terms $I_\infty$ and $L(\chi,2s-1)=\zeta_F(2s-1)$ to cancel their poles at $s=1/2$. The details will be left to the reader (see also \cite[5.5.3]{thesis}).


\subsection{Estimates off the critical line}

The estimates above are still valid for any real number $s$ (except if $\chi$ is trivial and $s=1$) but different exponents for $|\fri|$ and $\frf_\chi$ are obtained. We will not require tight estimates for those, and be content with stating the following rough result. 

\begin{prop}
For any $\sigma\in\RR$, $\chi\not=1$ and $\phi\in\sh(\chi,\tau)$ we have
$$
\|\Psi(\sigma)\phi\|_\sh,\left\|\frac{d\Psi(\sigma+iu)}{du}\phi \right\|_\sh \le C|\fri|^c \|\phi\|_\sh
$$
where $C$ depends on $F,\sigma$ and $\tau_\infty$ and $c$ only on $\sigma$. 
\label{estentreal}
\end{prop}

%% file: CMA.tex
The aim of this section is to define a Reidemeister torsion $\tau$ for congruence manifolds with cusps and then prove the following result. We fix a Bianchi group $\Gamma$ and a strongly acyclic $\Gamma$-module $V_\ZZ$. 

\begin{theo}
Let $\Gamma_n$ be a cusp-uniform sequence of pairwise distinct torsion-free congruence subgroups of $\Gamma$. Let $M_n=\Gamma_n\backslash\HH^3$ and let $\tau(M_n;V_\ZZ)$ be defined by \eqref{defRtors} below, then we have
$$
\lim_{n\to\infty}\frac{\log\tau(M_n;V_\ZZ)-\log T_R(M_n;V)}{\vol M_n} = 0. 
$$
\label{CMAfin}
\end{theo}

We now explain how this result follows from \cite{volI} and the results in \ref{cp} and \ref{ca} below. According to Theorem \ref{BSbianchi} and the proof of Theorem \ref{analytic} we can apply Theorem B in \cite{volI} to the sequence $M_n$, so that we get 
$$
\lim_{n\to\infty}\frac{\log\tau_\abs(M_n^{Y^n};V)-\log T_R(M_n;V)}{\vol M_n} = 0
$$
for the sequence $Y^n$ described there. According to Proposition \ref{compabs} below we can replace $Y^n$ by any $\Upsilon^n$ such that $\max\Upsilon_j^n\le|\fri_n|^c$ for some constant $c$, and the result now follows from Proposition \ref{compcomb} below. 

Before giving the definition of $\tau(M;V_\ZZ)$ and the proof of Proposition \ref{compcomb} we will recall from scratch how to describe analytically the cohomology of the boundary $\pl\ovl M$ of the Borel-Serre compactification $\ovl M$ of an hyperbolic manifold with cusps, and how to construct a section of the pull-back map $H^*(M)\to H^*(\pl\ovl M)$ using Eisenstein series as in \cite{Harder} (see also \cite[Section 3]{Berger}).


\subsection{Boundary cohomology} 
\label{cohbd}
We fix a congruence\footnote{Everything in the next three sections applies to all finite-volume hyperbolic three--manifolds.} manifold $M=M_{K'}$; the set of its cusps is in bijection with $\mathcal{C}(K')=C(F)\times(K_f'\bs K_f/N_f)$. Recall from \cite{volI} that $W_{l,k}$ is the subspace of $V_\CC$ of weight $(l,k)$ for the (complex) representation of $\LSL_2(\CC)\otimes\CC=\LSL_2(\CC)\oplus\LSL_2(\CC)$. In degree 1 we have an isomorphism
\begin{equation}
H^1(\pl M;V_\CC) \cong \CC[\mathcal{C}(K')]\otimes (W_{-n_1,n_2}\oplus W_{n_1,-n_2}) 
\label{deg1bd} 
\end{equation}
defined as follows: to a $2h$-tuple of vectors  $v_1,\ldots,v_h\in W_{-n_1,n_2},\bar v_1,\ldots,\bar v_h\in V_{n_1,-n_2}$ we associate the de Rham cohomology class $[\omega]$ of the 1-form $\omega$ given by
$$
\omega = \sum_{j=1}^h d\bar z_j\otimes (g_j \rho(n_{z_j})v_j) + dz_j\otimes (g_j \rho(n_{z_j})\bar v_j), \quad n_z = \begin{pmatrix} 1&z\\ &1\end{pmatrix}. 
$$
Let us check that $\omega$ is indeed a closed form. We have $v_j=w_j\otimes u_j$ where $w_j = g_j\lambda_j e_0$ and $u_j = g_j\ovl e_{n_2}$, so that $\rho(g_jn_z g_j^{-1}).v_j=w_j\otimes(\sum_{l=0}^{n_2}Q_l(\bar z)g_j\ovl e_{n_1-l}$ where $Q_l$ is a polynomial depending only on $n_2$. It follows that $z\mapsto\rho(n_z)v_j$ is anti-holomorphic. We can see in the same way that $z\mapsto \rho(n_z)\bar v_j$ is holomorphic and all this yields that 
$$
d(d\bar z_j\otimes(g_j\rho(n_{z_j})v_j)) = 0 = d(dz_j\otimes(g_j\rho(n_{z_j})\bar v_j)). 
$$
We will denote by $H^{1,0}(\pl\ovl M;V_\CC),H^{0,1}(\pl \ovl M;V_\CC)$ the subspaces of $H^1$ corresponding respectively to $\CC[\mathcal{C}(K')]\otimes W_{\mp n_1,\pm n_2}$. 

As for degrees 0 and 2 we have isomorphisms
\begin{equation}
H^0(\pl\ovl M;V_\CC) \cong \CC[\mathcal{C}(K')]\otimes W_{n_1,n_2} \cong H^2(\pl\ovl M;V_\CC). 
\end{equation}
Indeed, the space $W_{n_1,n_2}$ is the space of fixed vectors of $_0 N$ in $V$, and to $v_1,\ldots,v_h\in W_{n_1-n_2}$ we associate the holomorphic section $\sum_{j=1}^h g_j v_j$ or the holomorphic 2-form $\sum_{j=1}^h (dz_j\wedge d\bar z_j)\otimes(g_jv_j)$).


\subsection{Eisenstein cohomology}

The $L^2$-cohomology of $M$ with coefficients in $V_\CC$ vanishes and the map $i_p^* : H^p(M;V_\CC)\to H^p(\pl\ovl M;V_\CC)$ is thus an embedding for $p=1,2$ (cf. \cite[Theorem 2.1]{MFP2}). One can show using the long exact sequence of the pair $\ovl M, \pl\ovl M$ and Kronecker duality that 
\begin{equation}
\dim H^1(M;V_\CC)=1/2\dim H^1(\pl\ovl M;V_\CC) 
\label{eis=.5bd}
\end{equation}
(cf. \cite[Lemme 11]{Serrecong}). As  $H^0(M;V_\CC)=0$ the long exact sequence also yields that 
$$
\dim H^2(M;V_\CC)=\dim H^2(\pl\ovl M;V_\CC)-\dim H^0(M;V_\CC) = \dim H^2(\pl\ovl M;V_\CC).
$$
Now we will give an explicit description of the maps $i_p^*$ following \cite{Harder}. For a closed $p$-form $f\in\Omega^p(M;V_\CC)$ we denote by $[f]$ its de Rham cohomology class. Given an harmonic form $\omega\in H^1(\pl\ovl M;V_\CC)$ and a $s\in\CC$ we can form the Eisenstein series $E(s,\omega)\in\Omega^1(M;V_\CC)$. The following result is well-known, see for instance the proof of \cite[Theorem 2]{Harder}. 

\begin{lemma}
Let $s_V^1=n_2-n_1$ and $\omega\in H^{1,0}(\pl\ovl M;V_\CC)$ (resp. $\ovl\omega \in H^{0,1}(\pl\ovl M;V_\CC)$). The Eisenstein series $E(s_V^1,\omega)$ (resp. $E(-s_V^1,\ovl\omega)$) is then a closed 1-form. Moreover the classes $[E(s_V^1,\omega)]$ span $H^1(M;V_\CC)$. 
\label{deg1eis}
\end{lemma}

\begin{proof}
We need only check that if $P$ is a $\Gamma$-rational parabolic subgroup the constant term of $E(s_V^1,\omega)$ at $P$ is a closed form on $\Gamma_P\bs\HH^3$. It is equal (in the $\SL_2(\CC)$-equivariant model for $E_\rho$, see \cite[(2.5)]{volI}) to $\omega+\Phi^+(s_V^1)\omega$ and as $\omega,\Phi^+(s_V^1)\omega$ are  closed forms on $\pl\ovl M$ we have $d(\omega+\Phi^+(s_V^1)\omega)=0$. Moreover, we have $E(-s_V^1,\ovl\omega)=E(s_V^1,\Phi^-(-s_V^1)\ovl\omega)$ and thus the second statement follows. 

The constant term of $E(s_V^1,\omega)$ is also not an exact form since its restriction to $\pl\ovl M$ is not, and it follows that the map $H^{1,0}(\pl\ovl M;V_\CC)\ni\omega\mapsto[E(s_V^1),\omega]$ is injective. As we have the equality of dimensions 
$$
\dim H^{1,0}(\pl\ovl M;V_\CC)=1/2\dim H^1(\pl\ovl M;V_\CC)=\dim H^1(M;V_\CC)
$$
it is in fact an isomorphism.  
\end{proof}

We can thus define an application $E^1: H^1(\pl\ovl M;V_\CC)\to H^1(M;V_\CC)$ by $E^1(\omega+\ovl\omega)=[E(s_V^1,\omega)+E(-s_V^1,\ovl\omega)]$ for $\omega\in H^{1,0}(\pl\ovl M;V_\CC),\ovl\omega \in H^{0,1}(\pl\ovl M;V_\CC)$. From the formula for the constant term of Eisenstein series we get 
$$
i_1^*E^1(\omega)=\omega+\Phi^+(s_V^1)\omega,\quad \omega\in H^{1,0}(\pl\ovl M;V_\CC)
$$
and it follows that 
$$
\im i_1^* = \{\omega+\Phi^+(s_V^1)\omega,\, \omega\in H^{1,0}(\pl\ovl M;V_\CC)\} = \{ \Phi^-(-s_V^1)\ovl\omega+\ovl\omega,\quad \ovl\omega\in H^{0,1}(\pl\ovl M;V_\CC) \}. 
$$

In degree 2 the long exact sequence shows that $i_2^*$ is onto (since $H^3(\ovl M,\pl\ovl M;V_\CC)\cong H^0(M;V_\CC)=0$). We have a result akin to Lemma \ref{deg1eis} for this case, whose proof is very similar. 

\begin{lemma}
Let $s_V^0=n_1+n_2+1$ and $v\in V_N:= \bigoplus_{j=1}^h V_\CC^{N_j}\cong H^0(\pl\ovl M; V_\CC)$. The 2-form $*dE(s_V^0,v)$ is closed, and the classes $[*dE(s_V^0,v)]$ for $v\in V_N$ span $H^2(M;V_\CC)$. 
\label{deg2eis}
\end{lemma}

\begin{proof}
Computing the Casimir eigenvalue (cf. \cite[(2.4)]{volI}) one sees that $E(s_V^0,v)$ is harmonic, so that $d*dE(s_V^0,v)=0$. The constant term of $*dE(s_V^0,v)$ is a nonzero harmonic 2-form so that $[*dE(s_V^0,v)]$ is nonzero, and by equality of dimensions we get that these classes span $H^2(M;V_\CC)$. 
\end{proof}

We denote by $E^2$ the map $H^0(\pl\ovl M;V_\CC)\to H^2(M;V_\CC)$ defined by $v\mapsto [*dE(s_V^0,v)]$. 


\subsection{Inner products on cohomology and Reidemeister torsion}

From now on we will suppose that $V=V_{n_1,n_2}$ with $n_1>n_2$, so that $s_v^1\ge 1$. It follows from the Maass-Selberg relations \eqref{Maass-Selberg-real} that for $\omega\in H^{1,0}(\pl\ovl M; V_\CC)$ we have the limit 
$$
\lim_{Y\to\infty} Y^{-2s_V^1+1}\|T^YE(s_V^1,\omega)\|_{L^2(M)}^2 = (s_V^1)^{-1} \|\omega\|_{L^2(\pl\ovl M)}^2
$$
and we define an inner product on $H_\eis^1(M;V_\CC)$ by 
\begin{equation}
\begin{split}
\langle i_1^*[E^1(\omega)],i_1^*[E^1(\omega')]\rangle_{H_\eis^1(M)} &= \langle\omega,\omega'\rangle_{L^2\Omega^1(\pl\ovl M)}^2 \\
          &= \lim_{Y\to\infty} s_V^1\cdot Y^{-2s_V^1}\langle T^YE(s_V^1,\omega),T^YE(s_V^1,\omega')\rangle_{L^2\Omega^1(M)}^2. 
\end{split}
\label{deg1}
\end{equation}
Similarly, we can put 
\begin{equation}
\begin{split}
\langle i_2^*[E^2(v)],i_2^*[E^2(v')]\rangle_{H_\eis^2(M)} &= \langle v,v'\rangle_{L^2(\pl\ovl M)} \\
          &= \lim_{Y\to\infty} (s_V^0)^{\frac 1 2} Y^{-2s_V^0}\langle T^Y(*dE(s_V^0,v)),T^Y(*dE(s_V^0,v'))\rangle_{L^2(M)}. 
\end{split}
\label{deg2}
\end{equation}

Now for $p=1,2$ the integral cohomology $H^p(M;V_\ZZ)_\free$ is a lattice in the hermitian vector space $H^p(M;V_\CC)$. We finally define the Reidemeister torsion of $M$ with coefficients in $V$ by the formula
\begin{equation}
\tau(M;V_\ZZ) = \frac{|H^1(M;V_\ZZ)_\tors|}{\vol H^1(M;V_\ZZ)_\free} \times \frac{\vol H^2(M;V_\ZZ)_\free}{|H^2(M;V_\ZZ)_\tors|}.
\label{defRtors}
\end{equation}


\subsection{Asymptotic equality of Reidemeister torsions}
\label{cp}

We prove now that the Reidemeister torsion we just defined is asymptotically equal to the absolute Reidemeister torsion of the truncated manifolds (for a certain choice of truncations). 

\begin{prop}
Let $\Gamma_n,V$ be as in the statement of Theorem \ref{CMAfin}. There exists a sequence $\Upsilon^n$ such that 
\begin{equation}
\frac{\log\tau_\abs(M_n^{\Upsilon^n};V_\ZZ)-\log\tau(M_n;V_\ZZ)}{\vol M_n} \xrightarrow[n\to\infty]{} 0.
\end{equation}
and $\max_j \Upsilon_j^n \le |\fri|^c$ for some $c>0$. 
\label{compcomb}
\end{prop}

The first step is the following result, whose proof is essentially contained in \cite[6.8.3]{CV}. 

\begin{lemma}
There are $C,c>0$ depending only on $F$ such that the following holds. Let $\Gamma'\subset\Gamma$ be a congruence subgroup, $M=\Gamma'\bs\HH^3$, $h$ its number of cusps, $\alpha_j=\alpha_1(\Lambda_{n,j})$ where $\Lambda_{n,j},j=1,\ldots,h$ are the euclidean lattices corresponding to the cusps of $M'$. Then for all $Y\in[1,+\infty)^h$ such that for all $j$, $Y_j\ge C\alpha_j$,  $\omega\in H^{1,0}(\pl\ovl M;V_\CC), \, f=E(s_V^1,\omega)$ and $f_Y$ the projection of $f|_{M^Y}$ on the subspace $H_\abs^1(M^Y;V_\CC)$. Then we have
$$
\|f-f_Y\|_{L^2(M^Y)} \le C\|f\|_{L^2(M^Y)}e^{-c\min_j(Y_j/2\alpha_j)} \vol(M^Y-M^{Y/2}). 
$$
\label{cvagain}
\end{lemma}

\begin{proof}
Let $h:[1, +\infty[ \to [0,1]$ be a smooth function such that $h(1)=1, h(2)=0$ and define $f_Y'$ on $M^Y$ by $f_Y'=f-h(Y/y)(f-f_P)$ (where $y=\max_j y_j$). It follows from (6.16) of 
\cite{volI} that 
\begin{equation}
\|f-f_Y'\|_{L^2(M^Y)}  \le \|f-f_P\|_{L^2(M^Y-M^{Y/2})} \ll \|f\|_{L^2(M^Y)} e^{-c\min_j(Y_j/2\alpha_j)} . 
\label{majojo}
\end{equation} 
Now we check that $f_Y'$ satisfies absolute boundary conditions: close enough to the boundary we have $f_Y'=f_P$, and since $dy\wedge *f_P=0$ and $df_P=0$ we conclude that $f_Y'\in\Omega_\abs^1(M^Y;V_\CC)$. Thus, we have 
\begin{equation}
\Delta_\abs^1[M^Y] f_Y' = \Delta^1[M^Y] f_Y' = -\Delta^1[M^Y](h(Y/y)(f-f_P)) = (f_P-f)\Delta^1[M^Y]h(Y/y)
\label{majojomojo}
\end{equation} 
and the $L^2$-norm of the right-hand side is bounded by $C\|f\|_{L^2(M^Y)} e^{-c\min_j(Y_j/2\alpha_j)}$. 

According to the proof of Proposition 8.2 in \cite{volI}, up to making $C$ larger we may suppose that for $Y_j\ge C\alpha_j$ the Laplace operator $\Delta_\abs^1[M^Y]$ has no eigenvalue in the open interval $]0,\lambda_1[$ (for some $\lambda_1>0$ depending only on $V$) as soon as $Y_j\ge C\alpha_j$, and we then get from \eqref{majojo} and \eqref{majojomojo} that 
\begin{align*}
\|f-f_Y\|_{L^2(M^Y)} &\le \|f-f_Y'\|_{L^2(M^Y)} + \|f_Y'-f_Y\|_{L^2(M^Y)} \\ 
                &\le \|f\|_{L^2(M^Y)} \left(\int _{M^Y-M^{Y/2}} (\sum_{j=1}^h e^{-c\min_j(y_j(x)/\alpha_j)})^2dx\right)^{\frac 1 2} \\ 
                &\quad + \frac 2{\lambda_1} \|(f_P-f)\Delta^1[M^Y]h(Y/y)\|_{L^2(M^Y)} \\
                &\ll \|f\|_{L^2(M^Y)} \left(\int _{M^Y-M^{Y/2}} e^{-c\min_j y_j(x)/\alpha_j} dx\right)^{\frac 1 2} \\
                &\le \|f\|_{L^2(M^Y)} \vol(M^Y-M^{Y/2})e^{-c\min_j(Y_j/\alpha_j)}
\end{align*}
where the last line is a consequence of Cauchy-Schwarz inequality. 
\end{proof}

\begin{proof}[Proof of Proposition \ref{compcomb}]
Let $\Upsilon\in[1,+\infty)^{h_n}$; we have 
$$
\frac{\tau(M_n;V_\ZZ)}{\tau_\abs(M_n^\Upsilon;V_\ZZ)} = \frac{\vol H^2(M_n;V_\ZZ)_\free}{\vol H^2(M_n^\Upsilon;V_\ZZ)_\free} \frac{\vol H^1(M_n^\Upsilon;V_\ZZ)_\free}{\vol H^1(M_n;V_\ZZ)_\free } 
$$
and we will thus show that for $p=1,2$ we have 
$$
\log\vol H^p(M_n;V_\ZZ)_\free - \log\vol H^p(M_n^{\Upsilon^n};V_\ZZ)_\free = o(\vol M_n) 
$$
for a well-chosen sequence $\Upsilon^n$. 

We will deal only with $p=1$, the case $p=2$ being similar. Let $r_n$ be the restriction map $H^1(M_n;V_\CC)\to H^1(M_n^\Upsilon;V_\CC)$. As the inclusion $M_n^\Upsilon \subset M_n$ is an homotopy equivalence, it induces an isomorphism between the cohomology groups and we get that 
$$
\vol H^1(M_n^\Upsilon;V_\ZZ)_\free = |\det(r_n)|\vol H^1(M_n;V_\ZZ)_\free 
$$
where the determinant is taken with respect to unitary bases on each space (the left-hand space being endowed with the inner product defined by \eqref{deg1} and the right-hand on with the $L^2$ inner product coming from harmonic forms). We will show below that $\log|\det(r_n)|=o(\vol M_n)$, in fact that $|r_n|,|r_n|^{-1}\le 1+\eps_n$ for some sequence $\eps_n$ such that $b_1(M_n;V_\CC)\log\eps_n=o(\vol M_n)$). 

We take back the notation $f_\Upsilon$ from Lemma \ref{cvagain}, if $f$ is a closed form on $M_n$ we have $r_n[f]=[f_\Upsilon]$. To bound $\|f_\Upsilon\|_{L^2(M^Y)}$ above we write 
\begin{equation}
\|f_\Upsilon\|_{L^2(M^\Upsilon)}\le \|f\|_{L^2(M^\Upsilon)} + \|f-f_\Upsilon\|_{L^2(M^\Upsilon)} \le (1+C\sum_{j=1}^{h_n}\frac{\alpha_2(\Lambda_{n,j})}{\alpha_1(\Lambda_{n,j})})\|f\|_{L^2(M^\Upsilon)}
\label{brutasse}
\end{equation}
where the second inequality follows from Lemma \ref{cvagain} and the rough bound $\vol(M^\Upsilon-M^{\Upsilon/2})\le C\sum_{j=1}^{h_n}\frac{\alpha_2(\Lambda_{n,j})}{\alpha_1(\Lambda_{n,j})}$. Now we will bound the right-hand side using the following lemma. 

\begin{lemma}
Let $Y=\max_j \Upsilon_j$, $a=s_V^1$ and $\fri_n$ be the level of $\Gamma_n$. There are $b,C>0$ (depending on $F$ and $V$) such that 
\begin{equation}
(C^{-1}Y^a - C|\fri_n|^b)\|[f]\|_{H^1(M_n)} \le \|f\|_{L^2(M_n^\Upsilon)}\le C(Y^a+|\fri_n|^b)\|[f]\|_{H^1(M_n)}. 
\end{equation}
\label{commode}
\end{lemma}

\begin{proof}
Let $M_n^Y$ be the truncated manifold at height $(Y,\ldots,Y)\in[1,+\infty)^{h_n}$ so that $M_n^\Upsilon \subset M_n^Y$ and $\|f\|_{L^2(M_n^{Y^n})}  \le \|f\|_{L^2(M_n^Y)} \le \|T^Yf\|_{L^2(M_n)} $. The Maass-Selberg relations \eqref{Maass-Selberg-real} yield that 
\begin{align*}
\|T^Yf\|_{L^2(M_n)}^2 &\le \frac{Y^{2s_V^1-1}}{2s_V^1-1}\|\omega\|_{L^2\pl\ovl M}^2 + \frac{Y^{-2s_V^1+1}}{2s_V^1-1}\|\Phi^+(s_V^1)\omega\|_{L^2(\pl\ovl M)}^2 \\
               &\quad + \log Y \|\Phi^+(s_V^1)\omega\|_{L^1(\pl\ovl M)}\,\|\omega\|_{L^2(\pl\ovl M)} + \left\|\frac{d\Phi^+(s_V^1+iu)}{du}|_{u=0}\omega\right\|_{L^2(\pl\ovl M)}\,\|\omega\|_{L^2(\pl\ovl M)}.
\end{align*}
From proposition \ref{estentreal} it now follows that 
$$
\frac{\|T^Yf\|_{L^2(M_n)}^2}{|\omega|_{L^2(\pl\ovl M)}^2} \ll Y^{2s_V^1} + |\fri_n|^c (1+\log Y) \ll Y^{2a} + |\fri_n|^{2c}
$$
which deals with the upper bound; the lower bound is proved in a similar manner. 
\end{proof}

We have 
$$
\sum_{j=1}^{h_n}\frac{\alpha_2(\Lambda_{n,j})}{\alpha_1(\Lambda_{n,j})}\le |\fri_n| h_n\le h_F|K_f/N_{\fri_n}K_f(\fri_n)|\cdot|\fri_n| \le 2h_F|\fri_n|^3
$$
and it now follows from \eqref{brutasse} and Lemma \ref{commode} that for some $e>0$ we have 
\begin{equation}
\|f_\Upsilon\|_{L^2(M^\Upsilon)} \ll |\fri_n|^e Y^e 
\label{majorn}
\end{equation}
(we keep the notation $Y=\max_j\Upsilon_j$). 

The lower bound for $\|f_\Upsilon\|_{L^2(M^\Upsilon)}$ is more subtle. We have  
\begin{align*}
\|f_\Upsilon\|_{L^2(M^\Upsilon)} &\ge \|f\|_{L^2(M^\Upsilon)}-\|f-f_\Upsilon\|_{L^2(M^\Upsilon)} \\
              &\ge (1-\vol(M^\Upsilon-M^{\Upsilon/2})e^{-cY/\max\alpha_n^j}) \, \|f\|_{L^2(M^\Upsilon)} 
\end{align*}
where the second minoration follows from Lemma \ref{cvagain}. We have $\max\alpha_n^j\ll |\fri|^{\frac 1 2}$ and also $\vol(M^\Upsilon-M^{\Upsilon/2})\ll \sum_j \frac{\alpha_2}{\alpha_1}$ which is bounded by $|\fri_n|^3$, and it follows from Lemma \ref{commode} that 
\begin{equation}
\|f_\Upsilon\|_{L^2(M^\Upsilon)} \ge \left( 1 - C|\fri_n|^2\exp\left(-c\frac Y{|\fri_n|^{\frac 12}}\right)\right)(C^{-1}Y^a-C|\fri_n|^b)\|[f]\|_{H^1(M_n)}.
\label{minorn}
\end{equation}

For $A$ large enough and $\Upsilon_j^n=|\fri_n|^{A-1}$ we get from \eqref{majorn} and \eqref{minorn} that  
$$
1/2 \|[f]\|_{H^1(M_n)} \le \|f_Y\|_{L^2(M_n^{\Upsilon^n})} \le C|\fri_n|^{Ae}\|[f]\|_{H^1(M_n)}. 
$$
Thus $|r_n|^{-1}\le 2$ and $|r_n|\le C|\fri_n|^{Ae}$ and as $\dim H^1(M_n;V_\CC)=h_n$ it follows that 
$$
|\log\det(r_n)| \ll h_n\log|\fri_n|, 
$$
and as $h_n\ll(\vol M_n)^{1-\delta}$ (Lemma \ref{nbcusps}) the right-hand side is an $o(\vol M_n)$, as we wanted to show. 
\end{proof}


\subsection{Comparing absolute torsions}
\label{ca}
The following result is necessary to be able to use together Proposition \ref{compcomb} below and Theorem B in \cite{volI}, and its proof completes that of Theorem \ref{CMAfin}. 

\begin{prop}
Let $Y^n$ be the sequence from Theorem B of \cite{volI}. For any sequence $\Upsilon^n\in[1,+\infty)^{h_n}$ such that there is a $c>0$ for which $Y_j^n\le\Upsilon_j^n\le|\fri|^c$ we have 
$$
|\log\tau_\abs(M_n^{\Upsilon^n})-\log\tau_\abs(M_n^{Y^n})| \ll \dim H^*(M_n;V_\CC) \log|\fri_n|. 
$$
\label{compabs}
\end{prop}

\begin{proof}
It suffices to prove the result for $\Upsilon_j^n=|\fri_n|^c$. We will use  a smooth family of Riemannian metrics $g_u$, $u\in[1,+\infty)$ on $\ovl M$ such that 
\begin{itemize}
\item[(i)] $(\ovl M,g_u)$ is isometric to $M^u$ through a diffeomorphism $\phi_u$.
\item[(ii)] For $u/2\le v\le u$, $\phi_u\circ\phi_v^{-1}|_{M^{v/2}}$ is the inclusion map $M^{v/2}\subset M^u$.
\item[(iii)] Let $V$ be the line field perpendicular to horospheres (defined on $M-M^1$). We have 
\begin{equation}
\frac{dg_u}{du}\ll \frac 1 u g_u|_{V^\bot} + g_u|_V.
\label{ptm}
\end{equation}
\end{itemize}
Let us prove that such a family exists. We identify a collar neighbourhood $N$ of the boundary in $\ovl M$ with $\bigcup_j T_j\times[0,1]$ where the $T_j$ are the boundary components of $\pl M^1$. The metrics $g_u$ defined as follows do the job, as can be checked by an easy computation: on $\ovl M-N\cong M^1$ $g_u$ is the hyperbolic metric, an in the cusps we put
$$
|(v_1,v_2)|_{g_u}^2 = \frac 1{(u\, h(ut+1-u))^2}\left(|v_1|^2+u^2h'(ut+1-u)|v_2|^2\right), \quad v_2\in V_{(x,t)}, \, v_1\in V_{(x,t)}^\bot
$$
for $x\in T_j, t\in[0,1]$ where $h$ is a bump function which takes the value $1$ for $t\le 0$ and $1$ for $t\ge 1$ and we identify $V_{(x,t)}^\bot=T_xT_j$ and $V_{(x,t)}$ with the orthogonal complement of the latter in $T_xM^1$. 

Let $*_u$ be the Hodge star for $g_u$, put $\op=d*/du$ and $\alpha_u=*_u^{-1}\op \in \End_\CC\ker\Delta_\abs[M,g_u]$. Then we have \cite[Theorem 7.6]{RS} 
$$
\frac d{du}\log\tau_{abs}(M;g_u)=\tr(\alpha_u).
$$ 
Let $\lambda_u$ be the largest eigenvalue of $\alpha_u$, so that 
\begin{equation*}
|\log\tau_\abs(M^{Y^n})-\log\tau_\abs(M^{\Upsilon^n})| \le \sum_{j=1}^{h_n}\int_{Y_j}^{|\fri_n|^c} |\lambda_u|du.
\end{equation*}
Thus the result would follow if we proved that $|\lambda_u|\ll \frac 1 u$ for $u\ge Y_j$. First we compute the eigenvalues: if $f$ is an eigenform of $\alpha_u$ with norm 1 and eigenvalue $\lambda$ we have 
$$
\lambda=\frac{d\|v\|_{g_u}}{du}. 
$$
Indeed, $\langle *^{-1}\op f,f\rangle=\lambda$, so that we get $\lambda=\int_{\ovl M} \op f \wedge f = \frac d{du}\int_{\ovl M} *f\wedge f$. 

Now let $f$ be a harmonic 1-form for the metric $g_u$ which is an eigenform for $\alpha_u$; we want to see that $d|f|_{g_u}/du \ll u^{-1}$. On $M-M^1$ write $f=f_1+f_2$ according to the decompostion $TM=V\oplus V^\bot$ (in coordinates $f_1$ is the composant on $dy$), then according to \eqref{ptm} we have the pointwise inequality
$$
\left|\frac {d|f|_{g_u}}{du}\right| \ll |f_1|_{g_u}+u^{-1}|f_2|_{g_u} 
$$
so that we need to show that $|f_1|_{g_u}\ll u^{-1}$ on $M^u-M^{\frac u 2}$. The fact that $f$ is co-closed implies that $f_1$ has a vanishing constant term and it follows that 
$$
|f_1| = |f_1 - (f_1)_P| \le |f-f_P| \ll e^{-y_j/\alpha_1(\Lambda_j)}
$$
where the estimate is a consequence of \cite[Lemma 6.2.1]{CV}. The right-hand side is $\ll u^{-1}$ : indeed, the sequence $Y_j^n$ was defined in \cite{volI} as 
$$
Y_j^n = \alpha_1(\Lambda_{n,j}) \times\left(\frac{\vol M_n}{\sum_{j=1}^{h_n}(\alpha_2(\Lambda_{n,j})/\alpha_1(\Lambda_{n,j}))^2} \right)^{\frac 1{10}}
$$
and it follows from Lemma \ref{nbcusps} and the cusp--uniformity of the $M_n$ that $Y_j\gg \alpha_1(\lambda_j)|\fri_n|^\delta$ for some $\delta>0$. Thus, as we consider only $|\fri_n|^c\ge u\ge Y_j^n/2$ we get $\frac u{\alpha_1(\Lambda_j)}\gg u^\eta$ for some $\eta>0$ (depending on $\Upsilon_n$) and clearly $e^{-u^\eta}\ll u^{-1}$. 
\end{proof}

%% file: endgame.tex
We can now finish the proof of Theorem \ref{Main}, whose statement we recall below. 

\begin{theo}
Let $\Gamma$ be a Bianchi group, $\Gamma_n$ a cusp-uniform sequence of torsion-free congruence subgroups and $M_n=\Gamma_n\bs\HH^3$. Let $V$ be a real representation of $\SL_2(\CC)$ and $V_\ZZ$ a lattice in $V$ preserved by $\Gamma$. If $V$ is strongly acyclic then we have 
\begin{equation} \label{torsho}
\limsup_{n\to\infty}\frac{\log|H_1(\Gamma_n;V_\ZZ)_\tors|}{\vol M_n} \le -t^{(2)}(V). 
\end{equation}
and 
\begin{equation} \label{torsco}
\limsup_{n\to\infty}\frac{\log|H^2(\Gamma_n;V_\ZZ)_\tors|}{\vol M_n} \le -t^{(2)}(V). 
\end{equation}
\end{theo}

Let us describe how the results in this section articulate to yield this result. Recall that in \eqref{defRtors} we have defined a Reidemeister torsion for the congruence manifolds $M_n=\Gamma_n\bs\HH^3$ the logarithm of which is given by 
\begin{equation} \label{reidhere}
\begin{split}
\log\tau(M_n;V_\ZZ) &= \log|H^1(M_n;V_\ZZ)_\tors| - \log\vol H^1(M_n;V_\ZZ)_\free\\ 
                   &\quad + \log\vol H^2(M_n;V_\ZZ)_\free - \log |H^2(M_n;V_\ZZ)_\tors|.
\end{split}
\end{equation}
It follows from Theorems \ref{analytic} and \ref{CMAfin} that 
$$
\lim_{n\to\infty} \frac{\log\tau(M_n;V)}{\vol M_n} = t^{(2)}(V)
$$
and by Lemmas \ref{tors1}, \ref{reg2} and \ref{reg1} below all terms in \eqref{reidhere} except $-\log|H^2(M_n;V_\ZZ)_\tors|$ have a negative limit superior as $n\to\infty$. This proves \eqref{torsco}; we will deduce \eqref{torsho} from it in \ref{cohohoho} at the end of the section.


\subsection{Integral homology of the boundary} 

We have previously described the cohomology of the boundary with coefficients in $V_\CC$ using differential forms; to analyze the terms \eqref{reidhere} we will need a precise description of the integral homology and cohomology through cell complexes.

\subsubsection{Cell complexes}

Let $T$ be a 2--tori, $U$ a finite-rank free $\ZZ$-module with a representation $\rho:\pi_1(T)\to \SL(U)$. We fix a cell structure on $T$ with one 2--cell, $e^2$, two 1--cells $e_1^1,e_2^1$ and one 0--cell $e^0$ and denote by $u_1,u_2$ the associated basis for $\pi_1(T)$ (i.e. $u_i$ is the homotopy class of the loop $e_i^1$). Then we have an isomorphism of $\ZZ$-complexes $C_*(\wdt T;U)\cong C_*(\wdt T)\otimes U$ which yields an isomorphism of graded modules 
\[
C_*(T;U)=C_*(\wdt T;U)\underset{\ZZ[\pi_1(T)]}{\otimes}\ZZ\cong c_*(T)\otimes U.
\]
In this model the differentials for $C_*(T;U)$ are given by 
\begin{equation} \label{diff}
\begin{split}
d_2(e^2\otimes v)   &= e_1^1\otimes(v-\rho(u_2)v) + e_2^1\otimes(\rho(u_1)v-v),\\
d_1(e_i^1\otimes v) &= e^0\otimes(v-\rho(u_i)v).
\end{split}
\end{equation}


\subsubsection{Growth of torsion}

\begin{lemma}
Let $\Lambda$ be a lattice in a unipotent $F$-rational subgroup $N$, then for any sequence of pairwise distinct finite-index subgroups $\Lambda_n$ in $\Lambda$ we have 
$$
\log|H_0(\Lambda_n;V_\ZZ)_\tors|, \log|H_1(\Lambda_n;V_\ZZ)_\tors| = o([\Lambda:\Lambda_n]). 
$$
\label{torsbord}
\end{lemma}

\begin{proof}
We prove the result only for $\Lambda=1+\so_FX_\infty$ (where $X_\infty=\begin{pmatrix}0&1\\0&0\end{pmatrix}$), the general case can be reduced to that particular one. In the proof of Lemma \ref{sgoc} below we will show that if $1+aX_\infty\in\Lambda'$ then $(\Lambda'-1)V_\ZZ\supset Na\ovl V_\ZZ$, where  $\ovl V_\ZZ=\ker\rho(X_\infty)$. In particular, putting $d=\dim V$ we get: 
$$
|(V_\ZZ/(\Lambda'-1)V_\ZZ)_\tors| \le (N[\Lambda:\Lambda'])^d
$$
and the result about $H_0$ follows at once.  

Write now $\Lambda'=\ZZ u_1\oplus\ZZ u_2$. From \eqref{diff} we know that $H_1(\Lambda';V_\ZZ)$ embeds in $(V_\ZZ\oplus V_\ZZ)/\im(\rho(u_1)-1)\oplus(\rho(u_2)-1)$. The $\ZZ$-torsion of the latter itself embeds into 
$$
V_\ZZ/(\im(\rho(u_1)-1)\oplus V_\ZZ(\rho(u_2)-1)). 
$$
Now this last module has a torsion the order of which is bounded by $(N|u_1|^2\times N|u_2|^2)^d \ll [\Lambda:\Lambda']^{4d}$, and this finishes the proof for $H_1$. 
\end{proof}

\subsubsection{Free part of the homology}

Suppose now that $T$ has an Euclidean structure, so that its homology groups with coefficients in $V_\CC$ are endowed with the $L^2$ inner product and have a Hodge decomposition $H^1=H^{1,0}\oplus H^{0,1}$, which in the case of a boundary component of an hyperbolic manifold corresponds to the decomposition in \ref{cohbd}. We use the rational structure on the $F$-vector space $V_\QQ=\left(\sym^{n_1}F^2\right)\otimes\left(\sym^{n_2}\ovl{F^2}\right)$ given by restricting the scalars from $F$ to $\QQ$, which induces a rational structure on $H^*(\Gamma;V_\QQ)$; recall that a $\CC$-subspace $W\subset V_\CC$ is called rational when $\dim_\QQ(W\cap V_\QQ)=\dim_\CC W$.  

\begin{lemma}
\label{splitt}
The subspaces $H^{1,0}(T;V_\CC)$ and $H^{0,1}(T;V_\CC)$ are rational ; moreover there are constants $C,c$ depending only on $T,V$ such that for any finite cover $T'$ of $T$ of degree $D$, we have 
$$
[H^1(T';V_\ZZ) : H^{1,0}(T';V_\ZZ) \oplus H^{0,1}(T';V_\ZZ)] \le CD^c
$$
\end{lemma}

\begin{proof}
Let $z$ be a complex coordinate for $T'$ and put
$$
\omega_1 = d\ovl z\otimes(\rho(n_z)e_{0,n_2}), \quad \omega_2= dz\otimes(\rho(n_z)e_{n_1,0}), 
$$
then $\omega_1,\omega_2$ are generators (over $\CC$) for $H^{1,0}(T';V_\CC)$ and $H^{0,1}(T';V_\CC)$ respectively; we will check that they are rational and that they generate (over $\so_F$) a lattice whose index in $H^1$ is bounded by a polynomial in $D$. 

For that purpose we choose a basis $u_1,u_2$ for $\pi_1(T')$ such that $D \ge C_1|u_1| |u_2|$, where $C_1$ depends only on $T$. We will construct first a cycle 
\begin{equation}
\theta_2 = (P_0(u_1,u_2)e_1^1 + Q_0(u_1,u_2)e_2^1)\otimes \ovl e_{n_1,0} + \sum_{l=1}^{n_1} P_l(u_1,u_2)e_1^1\otimes e_{n_1-l,0}
\label{cycle1}
\end{equation}
where $P_i,Q_0$ are polynomials with integer coefficients depending only on $T,V$. We see that 
\begin{align}
\begin{split}
(\omega_1,\theta_2) &= 0 ; \\
(\omega_2,\theta_2) &= N^{-1}P(u_1,u_2)
\end{split}
\label{rel1}
\end{align}
where $P$ is an integral polynomial and $N$ an integer, both depending only on $T,V$. 

Let us show how to proceed to the construction of \eqref{cycle1}: write 
$$
\rho(z)e_{n_1-l,0} = e_{n_1-l,0} + \sum_{k=l+1}^{n_1} Q_k^l(z) e_{n_1-k,0}
$$
where $Q_k^l(z)\in\so_F[z]$ ; for ease of notation we identify $v=e^0\otimes v$, then we get that 
$$
d_1\left( (Q_1^0(u_2)e_1^1 - Q_1^0(u_1)e_2^1)\otimes e_{N_1,0}\right) = \otimes \sum_{l=2}^{n_1} P_l^1(u_1,u_2) e_{n_1-l,0}
$$ 
Let $\theta=(Q_1^0(u_2)e_1^1 - Q_1^0(u_1)e_2^1)\otimes e_{N_1,0}$, we further get that 
$$
d_1\left( \theta - \frac{P_2^1(u_1,u_2)}{Q_2^1(u_1)} e_1^1\otimes e_{n_1-1,0}  \right) = \sum_{l=3}^{n_1} \frac{P_l^2(u_1,u_2)}{Q_2^1(u_1)} e_{n_1-l,0}.
$$
Iterating this procedure until we reach $n_1$ we get a cycle whose coefficients are rational fractions in $u_1,u_2$, so that we only have to multiply by its denominator to get \eqref{cycle1}. 

In the same manner we can construct another cycle $\theta_1$ which satisfies 
\begin{align}
\begin{split}
(\omega_1,\theta_1) &= M^{-1}Q(u_1,u_2) ; \\
(\omega_2,\theta_2) &= 0
\end{split}
\label{rel2}
\end{align}
where $M,Q$ depend only on $T,V$. 

The index of the lattice $\so_F\theta_1\oplus\so_F\theta_2$ in $H_1(T';V_\ZZ)$ is bounded by a polynomial in $u_1,u_2$, which is itself bounded by a polynomial in $D$, say $R(D)$. On the other hand the computations \eqref{rel1},\eqref{rel2} show that $NR(D)\omega_2,MR(D)\omega_1$ are integral, and clearly the index of the lattice they span is bounded by a polynomial in $u_1,u_2$, which finishes the proof. 
\end{proof}

We now easily deduce :

\begin{lemma}
The covolumes $\vol H^p(T',V_\ZZ)_\free$ for $p=0,1,2,(0,1)$ and $(1,0)$ are all $o(\vol T')$ in a sequence of finite covers of $T$.
\label{volbd}
\end{lemma}

\begin{proof}
For $p=0,2$ this is trivial. For $p=(1,0), (0,1)$ we see that the $L^2$-norms of the integral classes constructed in the proof of Lemma \ref{splitt} are polynomially bounded in $\vol T'$, which implies the result in this case, and for $p=1$ we deduce it from the latter and Lemma \ref{splitt}. 
\end{proof}


\subsection{Subexponential growth of torsion} 

We prove here that in degrees other than $2$ the torsion in cohomology has subexponential growth. 

\subsubsection{Homology in degree 0}

\begin{lemma}
Let $\Gamma_n$ be a sequence of congruence subgroups in $\Gamma(\so_F)$, $M_n=\Gamma_n\bs\HH^3$. We have that $\log|H_0(M_n;V_\ZZ)|=o(\vol M_n)$ 
\label{sgoc}
\end{lemma}

\begin{proof}
We prove the result for principal congruence subgroups and then deduce the general case. To do the former we will show that $N\fri V_\ZZ\subset (\Gamma(\fri)-1)V_\ZZ$ for all $\fri$ and some integer $N$ depending only on $n_1,n_2$, so that $|H_0(M_\fri;V_\ZZ)|\le (N|\fri|)^{\dim V}$ from which it follows at once that $\log|H_0(M_\fri;V_\ZZ)|=O(\log|\fri|)$ is an $o(\vol M_\fri)$. 
 
Let $X_\infty=\begin{pmatrix} 0&1\\0&0\end{pmatrix}$ and $X_0=\begin{pmatrix} 0&0\\1&0\end{pmatrix}$. For $a\in\fri$ we have that $\eta_a=1+aX_\infty\in\Gamma(\fri)$. We begin by studying the case where $n_2=0,n_1=n$; put $e_1=(1,0)$ and $e_2=(0,1)$, then the family $e_1^n,e_1^{n-1}e_2,\ldots,e_2^n$ is an $\so_F$-basis of the free module $V_\ZZ$. Let $N$ be the product of all binomial coefficients $\binom{n}{k}$, we will see that $Nae_1^{n-k}e_2^k\in(\Gamma(\fri)-1)V_\ZZ$ for all $k<n$. Indeed, we have $ae_2^n=\eta_a.(e_1e_2^{n-1})-e_1e_2^{n-1}$, and on the other hand $\eta_a.(e_1^{k+1}e_2^{n-k-1})-e_1^{k+1}e_2^{n-k-1}$ is a linear combination of the $e_1^le_2^{n-l}$ for $l\ge k$ so we can prove this by induction on $k$. We also have that $ae_1^n=(1+aX_0)e_1^{n-1}e_2-e_1^{n-1}e_2\in(\Gamma(\fri)-1)V_\ZZ$, which finishes the proof in this case. The same arguments work in general. 

If $H\subset G_\fri$ is a proper subgroup there is an epimorphism $H_0(M_H;V_\ZZ)\to H_0(M_\fri;V_\ZZ)$. Letting $\fri_n$ be the level of $\Gamma_n$ we get that $\log|H_0(M_n;V_\ZZ)|=O(\log|\fri_n|)$, and it follows from this and Lemma \ref{index-level} that we have $\log|H_0(M_n;V_\ZZ)|=o(\vol M_n)$.
\end{proof}

\subsubsection{Cohomology in degree 1}

We will use the following elementary lemma in what follows.

\begin{lemma}
Let $A\in\hom(\ZZ^m,\ZZ^n)$ and $B\in\hom(\ZZ^n,\ZZ^m)$ such that for all $\varphi\in\hom(\ZZ^m,\ZZ)$ and $v\in\ZZ^n$ we have $(\varphi,Bv)=(\varphi\circ A,v)$. Then $\ZZ^m/B\ZZ^n$ and $\ZZ^n/A\ZZ^m$ have the same torsion subgroup.
\label{fudge}
\end{lemma}

\begin{proof}
In appropriate bases of $\ZZ^m,\ZZ^n$ the matrices of $A$ and $B$ are transpose of each other.
\end{proof}

\begin{lemma}
We have 
\begin{equation}
\frac{\log|H^1(M_n;V_\ZZ)_\tors|}{\vol M_n} \xrightarrow[n\to\infty]{} 0. 
\label{tors1}
\end{equation}
\end{lemma}

\begin{proof}
Recall that there is a $\Gamma(\so_F)$-invariant pairing on $V_\QQ$ and let $V_\ZZ'$ is the lattice in $V$ which is dual to $V_\ZZ$ through this pairing. For notational ease we will use $H_*,H^*$ to denote (co)homology with coefficients in $V_\ZZ$ and $H_*',H_{'}^*$ for $V_\ZZ'$-coefficients. The existence of the Kronecker pairing and the property \eqref{naturel}, together with Lemma \ref{fudge} imply that 
$$
[H_2(M)_\free : (\im i_*^2)_\free] = [H_{'}^2(\pl\ovl M)_\free : (\im i_2^*)_\free] 
$$
and it further follows that 
$$
[H_2(M)_\free : (\im i_*^2)_\free]  = [H_0'(\pl\ovl M)_\free : \im(\delta^1)_\free] \le |H_0'(M)| 
$$
where the equality follows from Poincar\'e duality and the majoration from the segment $H_1'(\ovl M,\pl\ovl M)\xrightarrow[]{\delta^1} H_0'(\pl\ovl M)\to H_0'(M)$ in the homology long exact sequence of the pair $\ovl M,\pl\ovl M$. Applying once more Poincar\'e duality we get 
\begin{equation}
[H^1(\ovl M,\pl\ovl M) : \im\delta_0] = [H_2(M)_\free : (\im i_*^2)_\free] \le |H_0'(M)_\tors|. 
\label{splorge}
\end{equation}
On the other hand the cohomology long exact sequence for $\ovl M,\pl\ovl M$ contains 
$$
H^0(\pl\ovl M) \xrightarrow[]{\delta_0} H^1(\ovl M,\pl\ovl M) \to H^1(M) \to H^1(\pl\ovl M)
$$
which in turn yields
\begin{align*}
\log|H^1(M)_\tors| &\le \log[H^1(\ovl M,\pl\ovl M):\im \delta_0] + \log|H^1(\pl\ovl M)_\tors| \\
                                 &\le \log|H_0'(M)_\tors| + \log|H_1(\pl\ovl M)_\tors|
\end{align*}
where the inequality on the second line follows from \eqref{splorge}. The right-hand side is an $o(\vol M_n)$, as follows from Lemmas \ref{sgoc} and \ref{torsbord}, which finishes the proof. 
\end{proof}


\subsection{Growth of regulators}

\subsubsection{Degree 1}

\begin{lemma} \label{reg1}
We have 
\[
\liminf_{n\to\infty}\frac{\log\vol H^1(M_n;V_\ZZ)}{\vol M_n} \ge 0. 
\]
\end{lemma}

\begin{proof}
The embedding $H^1(M_n;V_\CC)\to H^1(\pl\ovl M_n;V_\CC)$ is isometric by definition of the inner product on $H^1(M_n;V_\CC)$ and its image is the subspace 
$$
\{\omega+\Phi^+(s_V^1)\omega , \omega\in H^{1,0}(\pl\ovl M_n;V_\CC)\}. 
$$
Let $\pi$ be the orthogonal projection of $H^1(\pl\ovl M_n;V_\CC)$ onto $H^{1,0}(\pl\ovl M_n;V_\CC)$. Then Lemma \ref{splitt} implies that the image $\pi(H^1(\pl\ovl M_n;V_\ZZ))$ contains $H^{1,0}(\pl\ovl M_n;V_\ZZ)$ with an index which is $\ll (\vol M_n)^{4h_n}$. As 
$$
\vol\pi(i_1^*H^1(M_n;V_\ZZ)) \le \vol i_1^*H^1(M_n;V_\ZZ) 
$$
we get that 
$$
\vol H^1(M_n;V_\ZZ) \gg (\vol M_n)^{4h_n}\vol H^{1,0}(\pl \ovl M_n; V_\ZZ)
$$
and together with Lemmas \ref{volbd} and \ref{nbcusps} this finishes the proof. 
\end{proof}


\subsubsection{Degree 2}

\begin{lemma} \label{reg2}
We have 
\[
\frac{\log\vol H^2(M_n;V_\ZZ)_\free}{\vol M_n} \xrightarrow[n\to\infty]{} 0. 
\]
\end{lemma}

\begin{proof}
The map $i_2^* : H^2(M_n;V_\CC)\to H^2(\pl\ovl M_n;V_\CC)$ is an isometry according to the definition \eqref{deg2} of the inner product on $H^2(M_n;V_\CC)$. Moreover, using the long exact sequence we get 
$$
[H^2(\pl\ovl M_n;V_\ZZ) : \im i_2^*]  = |H^3(\ovl M_n,\pl\ovl M_n;V_\ZZ)| = |H_0(M_n;V_\ZZ)|
$$
and it follows that 
$$
|\log\vol H^2(M_n;V_\ZZ)_\free| \le \log|H_0(M_n;V_\ZZ)| + |\log\vol H^2(\pl\ovl M;V_\ZZ)|.
$$ 
Now the right-hand side is an $o(\vol M_n)$, as follows from Lemmas \ref{sgoc}\ref{volbd}, so that the proof is complete.
\end{proof}


\subsection{Homology from cohomology}
\label{cohohoho}
We can finally deduce \eqref{torsho} from \eqref{torsco}: from the sequence 
$$
H_1(\pl\ovl M) \to H_1(M) \to H_1(\ovl M,\pl\ovl M) \to H_0(\pl\ovl M),
$$
Lemmas \ref{torsbord} and \ref{sgoc}, and Poincar\'e duality it is clear that it suffices to show that the index of the sublattice $i_*H_1(\pl\ovl M_n)_\free$ in $H_1(M_n)_\free$ is an $o(\vol M_n)$. We will not detail how to prove this, as it follows from the proof of Lemma \ref{reg1} (where it was shown that the torsion subgroup of $H^1(\pl\ovl M_n)/i^*H^1(M_n)$ is of order $o(\vol M_n)$) and Kronecker duality as in the proof of Lemma \ref{tors1}. We could also have applied the universal coefficients theorem as in \cite[Lemma 3.1]{Pfaff} to deduce it from \eqref{torsco} applied to the dual lattice of $V_\ZZ$ in $V_\QQ$.